\documentclass[a4,reqno,11pt]{amsart}
\usepackage{amssymb}
\usepackage{amscd}
\usepackage{amsmath}
\usepackage{amsthm}
\usepackage{mathrsfs}
\usepackage{amssymb}
\usepackage[english]{babel}
\usepackage{cite}
\usepackage{allrunes}
\usepackage{color}
\newtheorem{thm}{Theorem}[section]
\newtheorem{lem}[thm]{\textbf Lemma}
\newtheorem{cor}[thm]{Corollary}

\newtheorem*{thma}{Theorem A}
\newtheorem*{thmb}{Theorem B}
\newtheorem*{thmc}{Theorem C}
\newtheorem*{thmd}{Theorem D}
\theoremstyle{definition}

\newtheorem{defn}[thm]{Definition}

\newtheorem{exe}[thm]{\textbf Example}

\theoremstyle{remark}

\newtheorem{rem}[thm]{Remark}

\renewcommand{\phi}{\varphi}
\renewcommand{\d}{\partial}
\newcommand{\dd}{\mathrm{d}}

\renewcommand{\R}{\mathbf{R}}
\newcommand{\C}{\mathbf{C}}

\renewcommand{\m}{\mathbf{m}}

\renewcommand{\S}{\mathscr{S}}
\newcommand{\supp}{\mathop{\rm supp}}
\newcommand{\sqrbrak}[1]{\left[#1\right]}
\newcommand{\norm}[1]{\left\Vert#1\right\Vert}
\newcommand{\set}[1]{\left\{#1\right\}}
\newcommand{\brkt}[1]{\left(#1\right)}
\newcommand{\abs}[1]{\left|#1\right|}
\renewcommand{\T}{\mathbf{T}}
\newcommand{\Z}{\mathbf{Z}}
\newcommand{\pr}[2]{\langle{#1,#2}\rangle}
\renewcommand{\p}[1]{\langle{#1}\rangle}

\begin{document}
\numberwithin{equation}{section}
\title[Multilinear Fourier integral operators]
{Global boundedness of multilinear Fourier integral operators}
\author{Salvador Rodr\'iguez-L\'opez}
\address{Department of Mathematics, Uppsala University, Uppsala, SE 75106,  Sweden}
\email{salvador@math.uu.se}
\author{Wolfgang Staubach}
\address{Department of Mathematics, Uppsala University, Uppsala, SE 75106,  Sweden}
\email{wulf@math.uu.se}
\keywords{Multilinear operators, Fourier integral operators}
\subjclass[2000]{35S30, 42B20, 42B99}
\noindent \thanks{The first author was supported by the EPSRC First Grant Scheme reference number EP/H051368/1 and is partially supported by grant MTM2010-14946.\\
The second author was partially supported by the EPSRC First Grant Scheme reference number EP/H051368/1}
\begin{abstract}
We study the global boundedness of bilinear and multilinear Fourier integral operators on Banach and quasi-Banach $L^p$ spaces, where the amplitudes of the operators are smooth or rough in the spatial variables. The results are obtained by proving suitable global boundedness of rough linear Fourier integral operators with amplitudes that behave as $L^{p}$ functions in the spatial variables. The bilinear and multilinear boundedness estimates are proven by using either an iteration procedure or decomposition of the amplitudes, and thereafter applying our global results for linear Fourier integral operators with rough amplitudes.
\end{abstract}
\maketitle

\section{Introduction and summary of the results}
The study of bilinear Fourier integral operators started quite recently and owes its initiation to the pioneering work of L.~Grafakos and M.~Peloso, \cite{MR2679898}.
In that paper the authors study the local boundedness of bilinear Fourier integral operators on Banach and quasi-Banach $L^p$ spaces. More specifically, they consider H\"ormander type amplitudes $a(x, y , z, \xi , \eta)\in C^{\infty}(\R^{5n}),$ satisfying the estimate
\begin{equation}\label{defn S zero one bilinear symbol}
|\d^{\alpha}_{x} \d^{\beta_1}_{y} \d^{\beta_2}_{z} \d^{\gamma_1}_{\xi} \d^{\gamma_2}_{\eta}a(x, y, z, \xi, \eta)|\leq C_{\alpha_1,\alpha_2, \beta_1, \beta_2,\gamma}(1 + |\xi| + |\eta|)^{m-|\gamma_1|-|\gamma_2|},
\end{equation}
for some $m\in \R$ and all multi-indices $\alpha_1,\, \alpha_2,\, \beta_1 ,\, \beta_2,\,\gamma$ in $\mathbf{Z}_{+}^n,$ and phase functions $\psi(x, \xi, \eta)\in C^{\infty} (\R^{n} \times \R^{n}\setminus {0} \times \R^{n}\setminus {0}),$ homogeneous of degree 1 jointly in $(\xi, \eta)$ variables. To these amplitudes and phases, Grafakos and Peloso associate the bilinear Fourier integral operator

\begin{equation}\label{defn GP FIO}
T_{a} (f , g)(x)= \int_{\R^{4n}} a(x,y,z,\xi,\eta)\, e^{i\psi(x, \xi, \eta)+i\langle x, \xi+\eta\rangle-i\langle y, \xi\rangle -i\langle z, \eta\rangle} \, f(y) g(z)\, dy\, dz\, d\xi\, d\eta.
\end{equation}
Moreover they assume that the amplitude is compactly supported spatially in $(x, y,z)$ variables and also supported frequency-wise in a set of the form $|\xi|\thickapprox|\eta|\thickapprox|\xi +\eta|,$ and the function $\varphi(x, \xi, \eta):= \langle x, \xi +\eta\rangle+\psi(x, \xi, \eta)$ satisfies the non-degeneracy conditions $\det\,(\d^{2}_{x, \xi} \varphi) \neq 0$ and $\det\,(\d^{2}_{x, \eta} \varphi) \neq 0,$ as well as the condition $|\nabla_{x} \varphi(x, \xi, \eta)| \thickapprox |(\xi , \eta)|$ on the support of $a(x, \xi , \eta).$ Here, the notation $A\thickapprox B$ means that there are constants $c_1$ and $c_2$ such that $ c_1 B\leq A\leq c_2 B.$ Under these conditions, Grafakos and Peloso showed that the Fourier integral operator $T_a$ of order $m=0,$  defined in \eqref{defn GP FIO}, is bounded from $L^{q_1}\times L^{q_2} \to L^{r}$ with $\frac{1}{q_1} +\frac{1}{q_2} =\frac{1}{r}$ and $2\leq q_{1}, \, q_{2}, \, r' \leq \infty.$ Furthermore, by keeping the spatial variables of the amplitude in a compact set but without any support assumption on the frequency variables, it was shown in \cite{MR2679898} that $T_a$ is bounded from $L^{1}\times L^{\infty} \to L^{1}$ and $L^{\infty}\times L^{1} \to L^{1},$ provided $m<-\frac{2n-1}{2}.$

From the point of view of the operators that are investigated in this paper, Grafakos and Peloso also considered Fourier integral operators where the phase function is of the form $\varphi_{1}(x, \xi)-\langle y, \xi\rangle+ \varphi_{2}(x, \eta)- \langle z, \eta\rangle$ and showed that the corresponding operators with compactly supported amplitudes are bounded from $L^{q_1}\times L^{q_2} \to L^{r}$ with $\frac{1}{q_1} +\frac{1}{q_2} =\frac{1}{r}$ and $1< q_{1}, \, q_{2} <2,$ provided that the order $m=-(n-1)\left((\frac{1}{q_1}-\frac{1}{2})+(\frac{1}{q_2}-\frac{1}{2})\right).$

Here we would also like to mention the interesting results obtained by F. Bernicot and P. Germain \cite{MR2680189} concerning boundedness of bilinear oscillatory integral operators of the form

\begin{equation*}
B_{\lambda} (f , g)(x)= \int_{\R^{4n}} a(x,\xi,\eta)\, e^{i\lambda\psi(\xi, \eta)+i\langle x, \xi+\eta\rangle-i\langle y, \xi\rangle -i\langle z, \eta\rangle} \, f(y) g(z)\, dy\, dz\, d\xi\, d\eta,
\end{equation*}
where $|\lambda|\geq 1,$ $a(x,\xi,\eta)$ satisfies an estimate similar to \eqref{defn S zero one bilinear symbol} but is not supposed to be compactly supported in the spatial variable $x,$ and the phase function $\psi(\xi, \eta)\in C^{\infty}(\R^n \times \R^n)$ satisfies suitable non-degeneracy conditions. In \cite{MR2680189}, Bernicot and Germain showed that the operator $B_{\lambda}$ is bounded from $L^{q_1}\times L^{q_2} \to L^{r}$ when $1<q_{1}, \, q_{2}, \, r \leq 2,$ and the bound does not exceed $|\lambda|^{\varepsilon}$ provided that the order $m=0,$ and $q_1,$ $q_2 ,$ $r$ and $\varepsilon$ all satisfy certain admissibility conditions.

Motivated by the work of D. Foschi and S. Klainerman concerning bilinear estimates for wave equations \cite{MR1755116} where operators similar to those that are considered here are investigated albeit in a different context, and also motivated by the work of C. Kenig and W. Staubach \cite{MR2357989} on the so called $\Psi-$pseudodifferential operators and the investigations initiated in N. Michalowski, D.Rule and W.Staubach in \cite{MRS} concerning bilinear pseudodifferential operators with $L^p$ spatial behaviour, we consider in this paper a class of bilinear Fourier integral operators and make a systematic study of their global boundedness. We shall also deal with the problem of boundedness of certain classes of multilinear Fourier integral operators.

A fact which we would like to highlight here is that our investigations in this paper serve as a motivation for studying rough operators i.e. Fourier integrals which are non-smooth in the spatial variables of their amplitudes. Indeed as we shall see later, the boundedness of rough linear operators can be used as an efficient tool in proving boundedness for smooth or rough multilinear operators.

Here and in the sequel we will use the shorthand notation FIO for Fourier integral operators. The multilinear FIOs studied in this paper are of the form
\begin{equation}\label{defn RS FIO}
T_a(f_1, \dots, f_N) (x)= \int_{\R^{2Nn}} a(x,\xi_1, \dots, \xi_N)\, e^{ i(\sum_{j=1}^{N} \varphi_{j}(x, \xi_{j})-\langle y_{j}, \xi_{j}\rangle)} \, \prod_{j=1}^{N}f(y_j)dy_1 \dots dy_N\, d\xi_1\dots d\xi_N,
\end{equation}
where the amplitude $a(x,\xi_{1},\dots,\xi_N)$ is assumed to be measurable in the spatial variable $x$ and smooth in the frequency variables $(\xi_{1}, \dots, \xi_N),$ satisfying the estimate
\begin{equation}\label{defn intro multilinear symbols}
\norm{\partial_{\xi_{1}}^{\alpha_1}\dots\partial_{\xi_{N}}^{\alpha_N}a(x,\xi_{1},\dots,\xi_N)}_{L^p}\leq C_{\alpha_{1}\dots \alpha_{N}} \p{\xi_{1}}^{m_1-\varrho_1\abs{\alpha_1}}\dots\p{\xi_{N}}^{m_N-\varrho_N\abs{\alpha_N}},
\end{equation}
for some $m_{1}, \dots, m_N \in \R$, $\varrho_{1}, \dots \, \varrho_N \in[0,1]$ $p\in [1,\infty]$ and all multi-indices $\alpha_1,\dots, \alpha_N$ in $\mathbf{Z}_{+}^n.$  Here the phase functions $\varphi_{j} (x,\xi_{j})$ are assumed to be $C^{\infty}(\R^n \times \R^n\setminus {0})$ and homogeneous of degree 1 in their frequency variables. Furthermore, we require that the phase functions verify the {\it{strong non-degeneracy conditions}} $|\det\,\d^{2}_{x, \xi} \varphi_{j}(x,\xi)| \geq c_{j} >0$ for $j=1,\dots, N.$

We shall study the boundedness of bilinear and multilinear FIOs separately. The main distinction between our bilinear and multilinear investigations is that, in proving the boundedness of the bilinear operators, we reduce matters directly to the case of linear FIOs with rough amplitudes. In fact, we establish global $L^{q}-L^{r}$ estimates for rough linear FIOs where the amplitudes are assumed to belong to the class defined in \eqref{defn intro multilinear symbols} with $N=2,$ and use this to prove the boundedness of bilinear FIOs. The global boundedness of linear FIOs is a problem of separate interest and our investigation here is somewhat related to the investigations of D. Dos Santos Fereirra and W. Staubach in \cite{DS}. In connection to the problem of global $L^p$ boundedness of FIOs, we should also mention the work of S. Coriasco and M. Ruzhansky in \cite{MR2677978} where the authors deal with global boundedness of FIOs with smooth amplitudes that belong to a subclass of $S^{0}_{1,0}$.

In the statements of the theorems below, we assume that the phase function belongs to the class $\Phi^2$ (see Definition \ref{defn Phik phases}) which requires certain control of the growth of the mixed derivatives of orders 2 and higher of the phase. Our linear global boundedness results are as follows:

\begin{thma}
 Suppose that $0<r\leq \infty$, $1\leq p,q \leq \infty$, satisfy the relation $1/r = 1/q + 1/p$. Let $a(x,\xi)$ verify the estimate in
 \eqref{defn intro multilinear symbols} with $N=1,$ $\varphi\in \Phi^2$ be a strongly non-degeneracy phase function and suppose further that $\varrho \leq 1,$ $s:=\min(2,p,q),$ $\frac{1}{s}+\frac{1}{s'}=1$ and
\begin{equation*}
	m < -\frac{(n-1)}{2}\brkt{\frac{1}{s} + \frac{1}{\min(s',p)}} + \frac{n(\varrho-1)}{s}.
\end{equation*}
Then the linear FIO
\[T_{a} u(x)= (2\pi)^{-n} \int_{\R^n} e^{i\phi(x,\xi)} a(x,\xi) \widehat{u}(\xi) \, \dd\xi,\]
 is bounded from $L^q$ to $L^r.$
\end{thma}

\begin{thmb}
Suppose that $0\leq \varrho\leq 1$, $p\geq 2$, $1\leq q\leq \infty$, $0<r\leq \infty$, satisfy $\frac{1}{r}=\frac{1}{p}+\frac{1}{q}$. Let $\varphi\in \Phi^2$ satisfy the strong non-degeneracy condition and $a(x,\xi)$ verify the estimate in
 \ref{defn intro multilinear symbols} with $N=1,$ and suppose that
$m<\textarc{m}(\varrho,p,q),$ with $\textarc{m}$ as in part $(a)$ of Definition \eqref{defn rune M}.\\
Then the operator $T_a$ defined in Theorem A above, is bounded from $L^q$ to $L^r.$ Furthermore, for $1<q<2$ and and $\mathcal{M}$ as in part $(b)$ of Definition \eqref{defn rune M} and
\begin{equation*}
\textarc{m}(\varrho,p,q)\leq m<\mathcal{M}(\varrho,p,q),
\end{equation*}
$T_a$ is bounded from $L^q$ to the Lorentz space $L^{r,q}.$
\end{thmb}

It is also important to note that the bounds occurring in the boundedness estimates in Theorems A and B depend only on $n$, $m$, $\varrho$, $p$, $q$, and a finite number of $C_\alpha$'s in Definition $\ref{defn intro multilinear symbols}$ with $N=1$.

Having the aforementioned linear theorems at our disposal, we can state and prove the following theorem which is one of our main results concerning bilinear FIOs.

\begin{thmc}
Assume that $q_1=\max(q_1,q_2)\geq p'$. Let
\[
    \frac{1}{r}=\frac{1}{p}+\frac{1}{q_1}+\frac{1}{q_2}.
\]
Then the bilinear FIO $T_a$, defined by \[T_{a} (f , g)(x)= \int_{\R^{4n}} a(x,\xi,\eta)\, e^{ i \varphi_{1}(x, \xi)+ i\varphi_{2}(x, \eta)-i\langle y, \xi\rangle-i\langle z, \eta\rangle} \, f(y) g(z)\, dy\, dz\, d\xi\, d\eta,\] with an amplitude satisfying \eqref{defn intro multilinear symbols} for $N=2$ verifies the estimate
\[
    \norm{T_a(f,g)}_{L^r}\leq C_{a,n} \norm{f}_{L^{q_1}}\norm{g}_{L^{q_2}},
\]
provided that
\[
    m_1<\textarc{m}(\varrho_1,p,q_1), \quad  m_2<\textarc{m}(\varrho_2,r_2,q_2) \quad {\rm and} \quad \frac{1}{r_2}=\frac{1}{p}+\frac{1}{q_1}.
\]
Moreover, if $1\leq q_2<2\leq r_2 ,$ $m_1<\textarc{m}(\varrho_1,p,q_1)$ and $\textarc{m}(\varrho_2,r_2,q_2)\leq m_2<\mathcal{M}(\varrho_2,r_2,q_2),$ then
\[
    \norm{T_a(f,g)}_{L^{r,q_2}}\leq C_{a,n} \norm{f}_{L^{q_1}}\norm{g}_{L^{q_2}}.
\]
\end{thmc}

Our second result which deals with multilinear FIOs extends a theorem in \cite{MR2679898} mentioned earlier concerning bilinear FIOs with phase functions of the form  $\varphi_{1}(x, \xi)-\langle y, \xi\rangle+ \varphi_{2}(x, \eta)- \langle z, \eta\rangle.$ We extend the aforementioned result to multilinear FIOs and to all ranges of parameters in the $L^{p}$ spaces and remove the assumption of compact spatial support on the amplitude. Furthermore, we also show a boundedness result concerning multilinear oscillatory integral operators without any homogeneity assumption on the phase.

\begin{thmd} Let $1\leq p\leq \infty$, $m_{j}<0,$ $j= 1,\dots N,$ and assume that $\frac{\sum_{j=1}^{N} m_{j}}{\min_{j=1,\dots,N}m_{j}}\geq \frac{2}{p}$. Assume that $a(x,\xi_{1},\dots, \xi_{N})$ satisfies the estimate \eqref{defn intro multilinear symbols} with $\varrho_1 =\dots= \varrho_{N}=1 .$
For $1\leq q_j < \infty$ in case $p=\infty,$ and $1\leq q_j \leq \infty$ in case $p\neq \infty,$ $j=1,\dots,N, $ let
\[
    \frac{1}{r}=\frac{1}{p}+\sum_{j=1}^{N}\frac{1}{q_j}.
\]
Then the multilinear FIO $T_a$, given by \eqref{defn RS FIO} and having strongly non-degenerate phase functions $\varphi_j\in\Phi^2 ,$ $j=1, \dots, N,$ satisfies the estimate
\[
    \norm{T_a(f_1,\dots, f_N)}_{L^r}\leq C_{a,n} \norm{f_1}_{L^{q_1}}\dots\norm{f_N}_{L^{q_N}},
\]
provided that
\[
    m_j<\textarc{m}(1,\frac{p(\sum_{k=1}^{N}m_k)}{m_j},q_j),\,{\rm for}\, j=1, \dots , N.
\]
In particular, if $a(x,\xi_1, \dots, \xi_N)$ verifies the estimate
\begin{equation*}
\norm{\partial_{\xi_1}^{\alpha_1}\dots\partial_{\xi_N}^{\alpha_N} a(\cdot,\xi_1,\dots,\xi_N)}_{L^{\infty}}\leq C_{\alpha_{1}\dots \alpha_{N}} (1+|\xi_1|+\dots+|\xi_N|)^{m-\sum_{j=1}^{N}\abs{\alpha_j}}.
\end{equation*}
Then $T_a$ is bounded from $L^{q_1} \times\dots\times L^{q_N} \to L^{r}$ provided that $\frac{1}{r}= \sum_{j=1}^{N} \frac{1}{q_j}$ and
\[
    m<-(n-1)\sum_{j=1}^{N}\abs{\frac{1}{q_j}-\frac{1}{2}}.
\]
 Moreover, if $m<0$ then $T_{a}$ is bounded from $L^2 \times \dots\times L^2 \to L^{\frac{2}{N}}$ provided that the phases $\varphi_{j}\in C^{\infty}(\R^n \times \R^n)$ are strongly non-degenerate and verify the condition $ |\partial_{x}^{\alpha} \partial^{\beta}_{\xi} \varphi_{j} (x, \xi)|\leq C_{j,\alpha,\beta}$ for $j=1,\dots, N$ and all multi-indices $\alpha $ and $\beta$ with $2\leq |\alpha|+|\beta|.$ In this case, no homogeneity of the phase in the $\xi$ variable is required.
 \end{thmd}
 It is worth mentioning that Theorem D also globalises and improves the order of the operator in the $L^\infty \times L^1 \to L^1$ boundedness proven in \cite{MR2679898} which was mentioned earlier. Namely, under our assumptions one can prove boundedness of a bilinear FIO from $L^\infty \times L^1 \to L^1$ and from $L^1 \times L^\infty \to L^1$ provided $m<-n+1$.

 Our results above are of some interest in connection to problems in partial differential equations. Indeed, our theorem applies to multilinear oscillatory integrals where the phase functions are of the form $\langle x, \xi_{j} \rangle + \psi_{j}(\xi_j)$, $j=1,\dots, N$ and the case of $\psi_{j}(\xi_j)= |\xi_j|$ which is homogeneous of degree 1, is relevant in connection to the study of the wave equation. Also, in the case of phases that are not homogeneous of degree 1 in the $\xi$ variable, we can for example obtain $L^2 \times L^2 \to L^1$ estimates for bilinear operators, where the case $\psi_{j}(\xi_j)= |\xi_j|^2$ with $\xi_{j} \in \R^n$ is related to the Schr\"odinger equation, $\psi_{j}(\xi_j)= \xi_j^3$ with $\xi_j \in \R$ corresponds to the Korteweg-de Vries equation, and $\psi_{j}(\xi_j)= \langle \xi_j \rangle$ with $\xi_{j} \in \R^n$ is related to the Klein-Gordon equation.
The proof of Theorem D uses a Coifman-Meyer type symbol decomposition as well as global boundedness results for linear FIOs, obtained here and in K. Asada and D. Fujiwara's paper \cite{AF}.
The structure of the paper is as follows. In Section \ref{prelim} we set up notations and basic definitions. In Section \ref{linear FIO} we use the Seeger, Sogge and Stein decomposition to decompose the linear FIO into low frequency and high frequency parts. Thereafter, following \cite{DS}, we establish the boundedness of the low frequency portion of the linear FIOs. Next, we turn to the main global $L^{q}-L^{r}$ estimates for rough linear FIOs. Finally in Section \ref{bilinear FIO} we treat the boundedness of bilinear as well as some multilinear FIOs and also give an application of some of the results to the boundedness of certain bilinear oscillatory integral operators.

\section{Notation, Definitions and Preliminaries} \label{prelim}

In this section we define the classes of linear and multilinear amplitudes with both smooth and rough spatial behaviour and also the class of phase functions that appear in the definition of the FIOs treated here.
\subsection{Classes of linear amplitudes}\label{subsec linear amplitudes}
In the sequel we use the notation $\langle \xi\rangle$ for $(1+|\xi|^2)^{\frac{1}{2}}.$ The following classical definition is due to H\"ormander \cite{H0}.
\begin{defn}\label{defn of hormander amplitudes}
Let $m\in \mathbf{R}$, $0\leq \delta\leq 1$, $0\leq \varrho\leq 1.$ A function $a(x,\xi)\in C^{\infty}(\mathbf{R}^{n} \times\mathbf{R}^{n})$ belongs to the class $S^{m}_{\varrho,\delta}$, if for all multi-indices $\alpha, \, \beta$
   it satisfies
   \begin{align*}
      \sup_{\xi \in \R^n} \langle \xi\rangle ^{-m+\varrho\vert \alpha\vert- \delta |\beta|}
      |\partial_{\xi}^{\alpha}\partial_{x}^{\beta}a(x,\xi)|< +\infty.
   \end{align*}
\end{defn}
We shall also deal with the class $L^{p}S^m_{\varrho}$ of rough symbols/amplitudes introduced by Michalowski, Rule and Staubach in \cite{MRS} which is the extension of the class of symbols introduced by Kenig and Staubach in \cite{MR2357989}.
\begin{defn} \label{LpSmrho definition}
Let $1 \leq p \leq \infty$, $m \in \R$ and $0\leq \varrho \leq 1$ be parameters. The symbol $a \colon \R^n \times \R^{n} \to \C$ belongs to the class $L^p S^m_{\varrho},$ if $a(x,\xi)$ is measurable in $x\in \R^n ,$ $a(x,\xi)\in C^\infty(\R^n_\xi)$ a.e. $x\in \R^n$, and for each multi-index $\alpha$ there exists a constant $C_\alpha$ such that
\begin{equation*}
\|\partial_\xi^\alpha a(\cdot,\xi)\|_{L^p (\R^n)} \leq C_\alpha \langle \xi\rangle^{m-\varrho |
\alpha|},
\end{equation*}
Here we also define the associated seminorms
\[
    \abs{a}_{p,m,s}=\sum_{\abs{\alpha}\leq s} \sup_{\xi\in \R^n}  \langle \xi\rangle^{\varrho |\alpha|-m} \norm{\partial_\xi^\alpha a(\cdot,\xi)}_{L^p (\R^n)}.
\]
\end{defn}
\begin{exe} If $b\in L^p$ and $\widetilde{a}(x,\xi)\in L^\infty S^m_\varrho$ then $a(x,\xi):=b(x)\widetilde{a}(x,\xi)\in L^pS^m_\varrho$. In particular, the same holds for $\widetilde{a}(x,\xi)\in S^{m}_{\varrho,\delta},$ with any $\delta.$
\end{exe}
\begin{exe}
 Take $\psi \in C_0^\infty$ with support in $[-1,1],$ and $h$ in the Zygmund class $L_{\rm exp}[-1,1]$ (see \cite[Chp. 4]{MR928802} for further details). Then $a(x,\xi):=e^{i\xi h(x)}\psi(x)\in L^pS^m_\varrho .$ In particular the amplitude $a(x,\xi)=e ^{i\xi \log |x|}\psi(x)$ belongs to $L^pS^0_0$. Observe that in this case, for every $x\neq 0$,  $a(x,\xi)\in C^\infty$ and $\Vert \partial^\alpha_\xi a(.,\xi)\Vert_{L^p}<\infty$ for all $p\neq \infty$, but for any $\alpha>0$, $\Vert\partial^\alpha_\xi a(\cdot,\xi)\Vert_{L^\infty}=+\infty$.

More generally, if $h, \psi$ are as above and $\sigma$ is a real valued function in $S^{m}_{\varrho,0}(\R^n)$  for $m\leq 0$ then $a(x,\xi)= e^{ih(x)\,\sigma(\xi)}\psi(x)$ is in the class $L^pS^m_\varrho$.
\end{exe}
\subsection{Classes of multilinear amplitudes}\label{subsec classes of multilinear amplitudes}
The class of multilinear H\"ormander type amplitudes is defined as follows:
\begin{defn}\label{defn multilinear hormander amplitudes}
Given $m\in \R$ and $\varrho, \delta \in [0,1],$ the amplitude $a(x,\xi_1,\dots,\xi_N)\in C^{\infty} (\R^n \times \R^{Nn})$ belongs to the multilinear H\"ormander class $S^{m}_{\varrho, \delta} (n, N)$ provided that for all multi-indices $\beta$, $\alpha_j$ $j=1, \dots, N$ in $\Z_{+}^{n}$ it verifies
\begin{equation}
    \abs{\partial^\beta_x \partial_{\xi_1}^{\alpha_1}\dots\partial_{\xi_N}^{\alpha_N} a(x,\xi_1,\dots,\xi_N)}\leq C_{\alpha_1,\dots,\alpha_N,\beta} \brkt{1+\abs{\xi_1}+\dots+\abs{\xi_N}}^{m-\varrho\sum_{j=1}^{N}\abs{\alpha_j}+\delta\abs{\beta}}.
\end{equation}

\end{defn}

We shall also use the classes of non-smooth amplitudes one of which is defined as follows:

\begin{defn}\label{defn multilinear product type amplitudes} Let $\m=(m_1,\dots, m_N)\in \R^{N}$ and ${\rho}=(\varrho_1,\dots,\varrho_N)\in [0,1]^N$. The symbol $a: \R^n\times \R^{Nn}\to \C$ belongs to the class $L_{\Pi}^p S^{\m}_{\rho}(n,N)$ if for all multi-indices $\alpha_1,\dots,\alpha_N$ there exists a constant $C_{\alpha_1, \dots, \alpha_N}$ such that
\begin{equation}\label{eq:Lp}
    \norm{\partial_{\xi_1}^{\alpha_1}\dots\partial_{\xi_N}^{\alpha_N} a(\cdot,\xi_1,\dots, \xi_N)}_{L^p}\leq C_{\alpha_1, \dots, \alpha_N} \p{\xi_1}^{m_1-\varrho_1\abs{\alpha_1}}\dots \p{\xi_N}^{m_N-\varrho_N\abs{\alpha_N}}.
\end{equation}
\end{defn}

We remark that the subscript $\Pi$ in the notation $ L_{\Pi}^p S^{\m}_{\rho}(n,2)$ is there to indicate the product structure of these type of amplitudes.
\begin{exe}
Any symbol in the class $m-S^{0}_{1,1}$ introduced by L. Grafakos and R. Torres in \cite{MR1880324}, is in $L_\Pi ^\infty S^{(0,\dots,0)}_{(1,\dots, 1)}(n,m)$.
\end{exe}

\begin{exe}
Let $a_j (x, \xi_j)\in L^{p_j}S^{m_j}_{\varrho_j}$ for $j=1,\dots, N$, be a collection of linear amplitudes and assume that $\frac{1}{p}=\sum_{j=1}^{N}\frac{1}{p_j}.$ Then the multilinear amplitude
$\prod_{j=1}^{N}a_j (x, \xi_j)$ belongs to the class $L_\Pi ^pS^{\m}_{\rho}(n,N).$
\end{exe}

Also we have the following class of non-smooth amplitudes introduced in \cite{MRS}.
\begin{defn}\label{defn multilinear MRS amplitudes} The amplitude $a: \R^n\times \R^{Nn}\to \C$ belongs to the class $L^p S^{m}_{\varrho}(n,N)$ if for all multi-indices $\alpha_1,\dots, \alpha_N$ there exists a constant $C_{\alpha_1, \dots, \alpha_N}$ such that
\begin{equation}
    \norm{\partial_{\xi_1}^{\alpha_1}\dots\partial_{\xi_N}^{\alpha_N} a(\cdot,\xi_1,\dots ,\xi_N)}_{L^p}\leq C_{\alpha_1, \dots, \alpha_N} \brkt{1+\abs{\xi_1}+\dots+\abs{\xi_N}}^{m-\varrho\sum_{j=1}^{N}\abs{\alpha_j}}.
\end{equation}
\end{defn}

\begin{exe}\label{interesting example}
It is easy to see that if $m\leq 0,$ $m_j\leq 0$ for $j=1, \dots, N$ and $p\in [1, \infty],$ then
\[
    L^p S^{m}_\varrho(n,N)\subset \bigcap_{m_1+\dots +m_N=m} L^p_\Pi S^{(m_1,\dots,m_N)}_{(\varrho,\dots, \varrho)}(n,N).
\]
Moreover for all $\varrho$ and $\delta$ in $[0,1]$
\[S^{m}_{\varrho, \delta}(n,N)\subset \bigcap_{m_1+ \dots+m_N=m} L^\infty_\Pi S^{(m_1,\dots,m_N)}_{(\varrho,\dots, \varrho)}(n,N).\]
\end{exe}

\begin{exe} Let $b\in S^m_{\varrho,0}(\R^{Nn})$ and $A$ be the matrix of a linear map from $\R^n$ in $\R^{Nn}$. Then
\[
    a(x,\xi_1,\dots, \xi_N)=b\,(Ax,\xi_1,\dots,\xi_N)\in S^m_{\varrho,0}(n,N).
\]

\end{exe}

\subsection{Classes of phase functions}\label{subsec classes of phases}
We also need to describe the class of phase functions that we will use in our investigation. To this end, the class $\Phi^{k}$  defined below, will play a significant role in our investigations.
\begin{defn} \label{defn Phik phases}
A real valued function $\phi(x,\xi)$ belongs to the class $\Phi^{k}$, if $\varphi (x,\xi)\in C^{\infty}(\R^n \times\R^n \setminus 0)$, is positively homogeneous of degree $1$ in the frequency variable $\xi$, and satisfies the following condition:

For any pair of multi-indices $\alpha$ and $\beta$, satisfying $|\alpha|+|\beta|\geq k$, there exists a positive constant $C_{\alpha, \beta}$ such that
   \begin{equation}\label{C_alpha}
      \sup_{(x,\,\xi) \in \R^n \times\R^n \setminus 0}  |\xi| ^{-1+\vert \alpha\vert}\vert \partial_{\xi}^{\alpha}\partial_{x}^{\beta}\phi(x,\xi)\vert
      \leq C_{\alpha , \beta}.
   \end{equation}
In connection to the problem of local boundedness of Fourier integral operators, one considers phase functions $\varphi(x,\xi)$ that are positively homogeneous of degree $1$ in the frequency variable $\xi$ for which $\det [\partial^{2}_{x_{j}\xi_{k}} \varphi(x,\xi)]\neq 0.$ The latter is referred to as the {\it{non-degeneracy condition}}. However, for the purpose of proving global regularity results, we require a stronger condition than the aforementioned weak non-degeneracy condition.

\end{defn}
\begin{defn}\label{strong non-degeneracy}
 A real valued phase $\varphi\in C^{2}(\R^n \times\R^n \setminus 0)$ satisfies the {\it{strong non-degeneracy condition}} or the {\it{SND condition}} for short, if there exists a positive constant $c$ such that
\begin{equation}\label{lower_bound on mixed hessian}
\Big|\det \frac{\partial^{2}\varphi(x,\xi)}{\partial x_j \partial \xi_k}\Big| \geq c,
\end{equation}
for all $(x,\,\xi) \in \R^n \times\R^n \setminus 0$.
\end{defn}

\begin{exe} A phase function intimately connected to the study of the wave operator, namely $\varphi(x,\xi)= |\xi|+ \langle x, \xi\rangle$ is strongly non-degenerate and belongs to the class $\Phi^2$.
\end{exe}

As is common practice, we will denote constants which can be determined by known parameters in a given situation, but whose
value is not crucial to the problem at hand, by $C$. Such parameters in this paper would be, for example, $m$, $\rho$, $p$, $n$ and the constants appearing in the definitions of various symbol classes. The value of $C$ may differ from line to line, but in each instance could be estimated if necessary. We also write sometimes $a\lesssim b$ as shorthand for $a\leq Cb$.

\section{Tools in proving boundedness of rough linear FIOs}

Here we collect the main tools in proving our boundedness results for linear FIOs. The following decomposition due to A. Seeger, C. Sogge and E. M. Stein is by now classical.
\subsection{The Seeger-Sogge-Stein decomposition}\label{sss decomposition}

One starts by taking a Littlewood-Paley partition of unity
\begin{equation}\label{eq:LittlewoodPaley}
    \Psi_0(\xi) +\sum_{j=1}^{\infty}\Psi_{j}(\xi)=1,
\end{equation}
where supp $\Psi_0\subset \{\xi;\,  \vert \xi \vert \leq 2 \}$, supp
$\Psi\subset \{\xi;\, \frac{1}{2} \leq \vert \xi \vert \leq 2 \}$
and $\Psi_{j}(\xi) =\Psi(2^{-j}\xi)$.

To get useful estimates for the amplitude and the phase function, one
imposes a second decomposition on the former Littlewood-Paley
partition of unity in such a way that each dyadic shell $2^{j-1}\leq
\vert \xi\vert\leq 2^{j+1}$ is partitioned into truncated cones of
thickness roughly $2^{\frac{j}{2}}$. Roughly $2^{\frac{(n-1)j}{2}}$
such elements are needed to cover the shell $2^{j-1}\leq \vert
\xi\vert\leq 2^{j+1}$. For each $j$ we fix a collection of unit
vectors $\{\xi^{\nu}_{j}\}_{\nu}$ that satisfy,
\begin{enumerate}
\item $\vert \xi^{\nu}_{j}-\xi^{\nu'}_{j}\vert\geq
2^{\frac{-j}{2}},$ if $\nu\neq \nu '$.
\item If $\xi\in\mathbb{S}^{n-1}$, then there exists a $
\xi^{\nu}_{j}$ so that $\vert \xi -\xi^{\nu}_{j}\vert
<2^{\frac{-j}{2}}$.
\end{enumerate}
Let $\Gamma^{\nu}_{j}$ denote the cone in the $\xi$ space whose
 central direction is $\xi^{\nu}_{j}$, i.e.
\begin{equation*}
    \Gamma^{\nu}_{j}=\{ \xi;\, \vert \frac{\xi}{\vert\xi\vert}-\xi^{\nu}_{j}\vert\leq 2\cdot 2^{\frac{-j}{2}}\}.
\end{equation*}
One can construct an associated partition of unity given by
functions $\chi^{\nu}_{j}$, each homogeneous of degree $0$ in $\xi$
and supported in $\Gamma^{\nu}_{j}$ with,
\begin{equation*}
\sum_{\nu}\chi^{\nu}_{j}(\xi)=1,\,\,\, \text{ for all}\,\,\, \xi \neq
0\,\,\, \text{and all}\,\,\, j
\end{equation*}
and
\begin{equation}
\label{chiestimate 1} \vert \partial
^{\alpha}_{\xi}\chi^{\nu}_{j}(\xi)\vert\leq C_{\alpha}
2^{\frac{\vert \alpha\vert j}{2}}\vert \xi\vert ^{-\vert \alpha
\vert},
\end{equation}
with the improvement
\begin{equation}
\label{chiestimate 2} \vert \partial
^{N}_{\xi_{1}}\chi^{\nu}_{j}(\xi)\vert\leq C_{N}
\vert \xi\vert ^{-N},
 \end{equation}
for $N\geq 1$.
Using $\Psi_{j}$'s and $\chi^{\nu}_{j}$'s, we can construct a
Littlewood-Paley partition of unity
\begin{equation*}
\Psi_{0}(\xi)+ \sum_{j=1}^{\infty}\sum_{\nu}\chi^{\nu}_{j}(\xi)\,
\Psi_{j}(\xi)=1.
\end{equation*}
Given a FIO
\begin{equation}\label{defn basic linear FIO}
  T_{a}u(x)= \frac{1}{(2\pi)^{n}}\int_{\mathbb{R}^{n}}e^{i
\varphi(x,\xi)}a(x,\xi)\hat{u}(\xi)
\, \dd\xi,
\end{equation}
we decompose it as
\begin{multline}
\label{Tdecomp}
T_0 u(x) + \sum_{j=1}^{\infty}\sum_{\nu}T^{\nu}_{j} u(x):=
\frac{1}{(2\pi)^{n}}\int_{\mathbb{R}^{n}}e^{i
\varphi(x,\xi)}a(x,\xi) \Psi_0(\xi)\hat{u}(\xi)
\, \dd\xi\\ +\frac{1}{(2\pi)^{n}}
\sum_{j=1}^{\infty}\sum_{\nu}\int_{\mathbb{R}^{n}}e^{i
\varphi(x,\xi)+i\langle x,\xi\rangle}a(x,\xi)\chi^{\nu}_{j}(\xi)
\Psi_{j}(\xi)\hat{u}(\xi)\, \dd\xi.
\end{multline}
We refer to $T_0$ as the {\it{low frequency part}}, and $T^{\nu}_{j}$ as the {\it{high frequency part}} of the FIO $T_{a}.$\\

Now, one chooses the axis in $\xi$ space such that $\xi_1$ is in the direction of $\xi^{\nu}_{j}$ and $\xi'=(\xi _2 , \dots , \xi_{n})$ is perpendicular to $\xi^{\nu}_{j}$ and introduces the phase function $\Phi(x,\xi):=
\varphi(x,\xi)-\langle (\nabla_{\xi}\varphi)(x,\xi^{\nu}_{j}),
\xi\rangle$ and the amplitude
\begin{equation}\label{amplitude}
A^{\nu}_{j}(x,\xi):= e^{i \Phi(x,\xi)}a(x,\xi)\chi^{\nu}_{j}(\xi)\,
\Psi_{j}(\xi).
\end{equation}

\noindent It can be verified, see e.g. \cite[p. 407]{MR1232192}
, that the phase
$\Phi(x,\xi)$ satisfies the following two estimates
\begin{eqnarray}
    \label{phaseestim1} \vert (\frac{\partial}{\partial \xi_1})^{N}
\Phi(x,\xi)\vert\leq C_{N} 2^{-Nj},\\
\label{phaseestim2} \vert (\nabla_{\xi'})^{N} \Phi(x,\xi)\vert\leq
C_{N} 2^{\frac{-Nj}{2}},
\end{eqnarray}
for $N\geq 2$ on the support of $A^{\nu}_{j}(x,\xi)$.

Using these, we can rewrite $T^{\nu}_{j}$ as a FIO with
a linear phase function,
\begin{equation}
\label{defn rewrittentnuj} T^{\nu}_{j} u(x)
=\frac{1}{(2\pi)^{n}}\int_{\mathbb{R}^{n}}A^{\nu}_{j}(x,\xi) e^{i
\langle (\nabla_{\xi}\varphi)(x,\xi^{\nu}_{j}),\, \xi\rangle} \hat{u}(\xi)\, \dd\xi.
\end{equation}

\subsection{Reduction of the phase function}\label{subsec phase reduction}

In this paper we will only deal with classes $\Phi^1,$ and more importantly $\Phi^2 ,$ of phase functions. In the case of class $\Phi^{2},$ we have only required control of those frequency derivatives of the phase function which are greater or equal to $2$. This restriction is motivated by the simple model case phase function $\varphi(x,\xi)=|\xi|+ \langle x,\xi\rangle$ for which the first order $\xi$-derivatives of the phase are not bounded but all the derivatives of order equal or higher than 2 are indeed bounded and so $\varphi(x,\xi)\in\Phi^2$. However in order to handle the boundedness of the low frequency parts of FIOs, one also needs to control the first order $\xi$ derivatives of the phase. The following phase reduction lemma will reduce the phase to a linear term plus a phase for which the first order frequency derivatives are bounded.

\begin{lem}\label{lem:change}
Any FIO $T_{a}$ of the type \eqref{defn basic linear FIO} with amplitude $a(x,\xi)\in L^{p}S^{m}_{\varrho}$ and phase function $\varphi(x,\xi)\in \Phi^2$, can be written as a finite sum of Fourier integral operators of the form
\begin{equation}\label{reduced rep of Fourier integral operator}
  \frac{1}{(2\pi)^{n}} \int a(x,\xi)\, e^{i\psi(x,\xi)+i\langle \nabla_{\xi}\varphi(x,\zeta),\xi\rangle}\, \widehat{u}(\xi) \, \dd\xi
\end{equation}
where $\zeta$ is a point on the unit sphere $\mathbf{S}^{n-1}$, $\psi(x,\xi)\in \Phi^1$ and $a(x,\xi) \in L^{p} S^{m}_{\varrho}$ is localized in the $\xi$ variable around the point $\zeta$.

\end{lem}
\begin{proof}
We start by localizing the amplitude in the $\xi$ variable by introducing an open convex covering $\{U_l\}_{l=1}^{M},$ with maximum of diameters $d$, of the unit sphere $\mathbf{S}^{n-1}$.
Let $\Xi_{l}$ be a smooth partition of unity subordinate to the covering $U_l$ and set $a_{l}(x,\xi)=a(x,\xi)\, \Xi_{l}(\frac{\xi}{|\xi|}).$ We set
\begin{equation}
  T_{l}u(x):= \frac{1}{(2\pi)^n} \int \, a_l(x,\xi)\, e^{i\varphi(x,\xi)}\,\widehat{u}(\xi) \, \dd\xi,
\end{equation}
and fix a point $\zeta \in U_l.$  Then for any $\xi\in U_l$, Taylor's formula and Euler's homogeneity formula yield
\begin{equation}
  \varphi(x,\xi) = \varphi(x,\zeta) + \langle \nabla_{\xi}\varphi (x,\zeta), \xi-\zeta\rangle +\psi (x,\xi)= \psi(x,\xi)+\langle \nabla_{\xi}\varphi(x,\zeta),\xi\rangle
\end{equation}

Furthermore, for $\xi\in U_l$, $\partial_{\xi_k} \psi(x,\xi)= \partial_{\xi_k} \varphi(x,\frac{\xi}{|\xi|})-\partial_{\xi_k} \varphi(x,\zeta)$, so the mean value theorem and the definition of class $\Phi^2$ yield $|\partial_{\xi_k} \psi(x,\xi)|\leq Cd$ and for $|\alpha|\geq 2$, $|\partial^{\alpha}_{\xi} \psi(x,\xi)|\leq C |\xi|^{1-|\alpha|}.$ Here we remark in passing that in dealing with function $\psi(x,\xi),$ we only needed to control the second and higher order $\xi-$derivatives of the phase function $\varphi(x,\xi)$ and this gives a further motivation for the definition of the class $\Phi^2 .$ We shall now extend the function $\psi(x,\xi)$ to the whole of $\mathbf{R}^{n}\times \mathbf{R}^{n}\setminus 0$, preserving its properties and we denote this extension by $\psi(x,\xi)$ again. Hence the Fourier integral operators $T_l$ defined by
\begin{equation}\label{eq:chng_var}
  T_{l}u(x):=\frac{1}{(2\pi)^{n}} \int a_l(x,\xi)\,e^{i\psi(x,\xi)+i\langle \nabla_{\xi}\varphi(x,\zeta),\xi\rangle}\, \widehat{u}(\xi) \, \dd\xi,
\end{equation}
are the localized pieces of the original Fourier integral operator $T_{a}$ and therefore $T=\sum_{l=1}^{M} T_l$ as claimed.
\end{proof}

\subsection{A uniform non-stationary phase estimate}\label{subsec uniform non stat phase}
We will also need a uniform non-stationary phase estimate that yields a uniform bound for certain oscillatory integrals that arise as kernels of certain  operators. To this end, we have:
\begin{lem}\label{lem:technic} Let $\mathcal{K}\subset \R^n$ be a compact set, $U\supset \mathcal{K}$ an open set and $k$ a nonnegative integer. For $u\in C^\infty_0(\mathcal{K})$ and $f$ a real valued function in $C^{\infty}(U),$ assume that $\abs{\nabla f}>0$ and for all multi-indices $\alpha$ with $\abs{\alpha}\geq 1,$ $\Psi$ satisfies the following estimates
\[
    \abs{\d^\alpha f}\lesssim \abs{\nabla f},\qquad  \abs{\d^\alpha \Psi}\lesssim \abs{\nabla f}^2.
\]
Then for any integer $k\geq 0$
\begin{equation*}
	\lambda ^k \abs{\int_{\mathbf{\R ^n}} u(\xi)\, e^{i\lambda f(\xi)}\, \dd \xi}\leq C_{k,n,\mathcal{K}} \sum_{\abs{\alpha}\leq k} \int_\mathcal{K} \abs{\d^{\alpha}_{\xi} u(\xi)} \abs{\nabla_{\xi} f(\xi)}^{-k}\, \dd \xi, \quad \lambda>0.
\end{equation*}
\end{lem}
\begin{proof}
    Let $\Psi=\abs{\nabla f}^2$. Let us define $A_0=u$ and
    \[
        A_{k}^{j_1,\ldots,j_k}=\d_{j_l}\brkt{A^{j_1,\ldots,j_{k-1}}_{k-1} \frac{\d_{j_l} f}{\Psi}},
    \]
    for $k\geq 1$, $j_l\in \set{1,\ldots,n}$ for $l\in\{0,\ldots,k\}$.

    We claim that for any multi-index $\alpha$ with $\abs{\alpha}\geq 0$, $\abs{\partial^\alpha\brkt{\Psi^{-1}}}\lesssim \Psi^{-1}$. Using induction on $|\alpha|$, we trivially have $\abs{\partial^0 \Psi}=\abs{\Psi},$ and so as our induction hypothesis, we assume that $\abs{\alpha}\geq 1$ and $\abs{\d^\gamma \Psi^{-1}}\lesssim \Psi^{-1}$ for any multi-index $\gamma$ with $\abs{\gamma}<\abs{\alpha}.$ Since $1=\Psi \Psi^{-1}$ Leibniz rule yields
    \[
        \partial^\alpha\brkt{\Psi^{-1}} \Psi=-\sum_{\beta<\alpha}\binom{\alpha}{\beta}  \partial^\beta\brkt{\Psi^{-1}}  \partial^{\alpha-\beta}\brkt{\Psi},
    \]
   from which, our induction hypothesis and the assumption on $\Psi$ in the statement of the lemma, the claim follows.
    Observe that, for any multi-index $\alpha$, $\abs{\alpha}\geq 0$,
    \[
        \begin{split}
        \d^\alpha \brkt{A_{k}^{j_1,\ldots,j_k}}
        &=\sum \binom{\alpha}{\beta}\binom{\beta}{\gamma} \Bigg( \d^\beta \d_{j_k}A_{k-1}^{j_1,\ldots,j_{k-1}}\, \d^\gamma \d_{j_k}f \, \d^{\alpha-\beta-\gamma}\brkt{\Psi^{-1}}\\
        &+\d^\beta A_{k-1}^{j_1,\ldots,j_{k-1}}\, \d^\gamma \d_{j_k,j_k}^2f \,\d^{\alpha-\beta-\gamma}\brkt{\Psi^{-1}}\\
        &+\d^\beta \d_{j_k} A_{k-1}^{j_1,\ldots,j_{k-1}}\, \d^\gamma \d_{j_k}f \, \d^{\alpha-\beta-\gamma}\d_{j_k}\brkt{\Psi^{-1}}\Bigg).
        \end{split}
    \]
    Proceeding by induction, one can see that for $k\geq 1$ and for any multi-index $\alpha$ with $\abs{\alpha}\geq 0$,
    $A_k^{j_1,\ldots,j_k}\in C^{\infty}_{0}(\mathcal{K})$ and
    \begin{equation}\label{eq:techn}
        \abs{\partial^\alpha A_k^{j_1,\ldots,j_k}}\lesssim \sum_{\abs{\beta}\leq \abs{\alpha}+k} \abs{\d^\beta u} \Psi^{-k/2}.
    \end{equation}
    Since $1=\sum_{j=1}^n \frac{\partial_{j}f}{\Psi} \partial_j f$, and $i \lambda \partial_jf e^{i\lambda f}= \d_{j}\brkt{e^{i\lambda f}}$, integration by parts yields
    \[
        (-i\lambda)^k\int_{\R ^n} u(\xi) e^{i\lambda f(\xi)}\, \dd \xi=\sum_{j_1,\ldots,j_k=1}^n \int_{\mathcal{K}} A^{j_1,\ldots,j_k}_k e^{i\lambda f(\xi)}\, \dd \xi.
    \]
    Then the result follows by taking absolute values of both sides and using \eqref{eq:techn} for $\abs{\alpha}=0$.
\end{proof}
\section{Global $L^q -L^r$ boundedness of rough linear FIOs }\label{linear FIO}
In this section we shall state and prove a boundedness result
concerning certain classes of FIOs with rough amplitudes and smooth strongly non-degenerate phase
functions.

\subsection{Boundedness of the low frequency part of the FIO}\label{subsec low freq}
Using the Seeger-Sogge-Stein decomposition from subsection \ref{sss decomposition}, here we shall establish the boundedness of the low frequency portion of the Fourier integral operator given by

\[T_0 u(x)=\frac{1}{(2\pi)^{n}}\int_{\mathbb{R}^{n}}e^{i
\varphi(x,\xi)}a(x,\xi) \Psi_0(\xi)\hat{u}(\xi)
\, \dd\xi,\]
where $\Psi_0 \in C_{0}^{\infty}$ and is supported near the origin. Clearly, instead of studying $T_0$, we can consider a FIO $T_{a}$ whose amplitude $a(x,\xi)$ is compactly supported in the  frequency variable $\xi.$ In what follows, we shall adopt this and drop the reference to $T_0$.
But before, we proceed with the investigation of the $L^q-L^r$ boundedness, we will need the following lemma.

\begin{lem} \label{lem:fuijiwara}
 Let $\eta(\xi)$ be a $C_{0}^\infty$ function and set
\[
	K(x,z):=\int_{\R^n} \eta(\xi) e^{i(\psi(x,\xi)+\pr{z}{\xi})}\, \dd \xi,
\]
where $\psi(x,\xi)\in \Phi^1 .$
Then, for any $\alpha\in (0,1)$, there exists a positive constant $c$ such that
\[
	\abs{K(x,z)}\leq c (1+\abs{z})^{-n-\alpha}.
\]
\end{lem}
\begin{proof}
The proof is a straightforward application of Lemma $1.2.10$ in \cite{DS}.
\end{proof}

\begin{thm}
\label{Linearlow}
Let $T_a$ be a FIO given by \eqref{defn basic linear FIO}, with a phase function $\varphi (x,\xi)\in \Phi^2 $ satisfying SND, and with a symbol
$a(x,\xi)\in L^pS^m_\varrho$ such that $\supp_{\xi}a (x,\xi)$ is compact. Suppose that $0<r\leq \infty$, $1\leq p,q\leq \infty$ satisfy the relation $1/r = 1/q + 1/p$.
Then the operator $T_a$ is bounded from $L^q$ to $L^r$ with norm bounded by a constant depending only on $n$, $m$, $p$, $q$, and a finite number of $C_\alpha$'s in Definition $\ref{LpSmrho definition}$.
\end{thm}
\begin{proof}
Consider a closed cube $Q$ of side-length $L$ such that $\supp_{\xi} a(x,\xi) \subset \text{Int}(Q).$ We extend $a(x,\cdot)|_{Q}$ periodically with period $L$ into $\widetilde{a}(x,\xi)\in C^{\infty}(\R^{n}_\xi).$ Let $\eta (\xi)$ be in $C^{\infty}_{0}$ with $\supp \eta \subset Q$ and $\eta=1$ on $\xi$-support of $a(x,\xi)$. Clearly, we have $a(x,\xi)=\widetilde{a}(x,\xi) \eta(\xi)$. Now if we expand $\widetilde{a}(x,\xi) $ in a Fourier series, then setting $u_k(x)=u(x-\frac{2\pi k}{L})$ for any $k\in \Z^n ,$ we can write the FIO $T_a$ as
\begin{equation}\label{eq:Fourier_Serie}
	T_a u(x)=\sum_{k\in \Z^n} a_k(x) T_\eta (u_k)(x),
\end{equation}
where
\[
	a_k(x)=\frac{1}{L^n}\int_{\R^n} a(x,\xi) e^{-i \frac{2\pi}{L}\pr{k}{\xi}}\, \dd \xi,
\]
and $T_\eta(v)(x):=\frac{1}{(2\pi)^n}\int \eta(\xi) e^{i \varphi(x,\xi)}\widehat{v}(\xi)\, \dd \xi.$
Let us assume for a moment that $T_\eta$ is a bounded operator on $L^q$. Then integration by parts yields
\[
	a_k(x)= \frac{c_{n,N}}{|k_{l}|^N}\int_{\R^n} \partial^N_{\xi_l} a(x,\xi) e^{-i  \frac{2\pi}{L}\pr{k}{\xi}}\, \dd \xi.
\]
Observe also that, by the hypothesis on the symbol and Lemma \ref{lem:calculus_lemma},
\begin{equation*}
	\max_{s=0,\ldots, N}\int_{\R^n} \norm{\partial^s_{\xi_{l}} a(\cdot,\xi)}_{L^p}\, \dd\xi\leq c_{n, N,\varrho}\abs{a}_{p,m,N}.
\end{equation*}
Thus 
\begin{equation}\label{estimak}
	\norm{a_k}_{L^p}\leq c_{n, N,\varrho}\abs{a}_{p,m,N} (1+\abs{k})^{-N}.
\end{equation}
Let us first assume that $r\geq 1$. Then by the Minkowski and H\"older inequalities,

\begin{equation}\label{eq:bound}
\begin{split}
	\norm{T_a u}_{L^r}\leq  \sum_{k\in \Z^n}  \norm{a_k T_{\eta}(u_k)}_{L^r}
\leq  \sum_{k\in \Z^n} \norm{a_k}_{L^p} \norm{T_{\eta}(u_k)}_{L^{q}}.
\end{split}
\end{equation}

On the other hand, since we have assumed that $T_\eta$ is bounded on $L^q$ and the translations are isometries on $L^q$, we have that $\norm{T_\eta (u_k)}_{L^q}\leq  c_{\eta,\varphi} \norm{u}_{L^q}$. Therefore using \eqref{estimak}
\[
	\norm{T_a u}_{L^r}\lesssim \abs{a}_{p,m,N} \sum_{k\in \Z^n} (1+|k|)^{-N} \norm{u}_{L^q}.
\]
Then selecting $N=n+1$ we conclude the proof.

Assume now that $0<r<1$. Using \eqref{eq:Fourier_Serie} and H\"older's inequality, with exponents $p/r$ and $q/r$,  we have 
\[
	\int \abs{T_a u(x)}^r \, \dd x
	\leq \sum_{k\in \Z^n}  \int \abs{T_\eta (u_k)(x)}^r\abs{a_k(x)}^r\, \dd x\leq 	\sum_{k\in \Z^n}\norm{a_k}_{L^p}^r \norm{T_{\eta}(u_k)}_{L^{q}}^r.
\]
The boundedness assumption on $T_\eta$ and \eqref{estimak} yields
\[
  \int \abs{T_a u(x)}^r \, \dd x 	\lesssim \abs{a}_{p,m,N}^r \sum_{k\in \Z^n} (1+|k|)^{-Nr} \norm{u}_{L^q}^r.
\]
Then, selecting $N=[n/r]+1$, we obtain the result.

In order to finish the proof we have to show that $T_\eta$ defines a bounded operator on $L^q$, for $1\leq q\leq \infty$. By Lemma \ref{lem:change} we can assume without loss of generality that
\[
	\varphi(x,\xi)=\psi(x,\xi) + \langle \textbf{t}(x), \xi\rangle,
\]
with a smooth map $\textbf{t}: \R^{n}\to \R^n ,$ stratifying $\abs{{\rm det}\, D \textbf{t}(x)}\geq c>0$ as a direct consequence of our SND assumption on the phase function $\varphi,$ and $\psi(x,\xi)\in \Phi^1.$  Furthermore, it follows from Schwartz's global inverse function theorem (see \cite[Theorem 1.22]{MR0433481}), that the map $x\mapsto \mathbf{t}(x)$ is a global diffeomorphism on $\R^n$.

For $v\in \S$ one has
\begin{equation}\label{eq:kernel_def}
	T_\eta(v)(x)=\frac{1}{(2\pi)^n}\int \eta(\xi) e^{i \pr{\xi}{\mathbf{t}(x)}} e^{i \psi(x,\xi)}\widehat{v}(\xi)\, \dd \xi=\int  K(x,\mathbf{t}(x)-y) v(y)\, \dd y,
\end{equation}
with
\begin{equation}\label{low frequency kernel estim}
	K(x,z)=\frac{1}{(2\pi)^n}\int \eta(\xi) e^{i \pr{\xi}{z}} e^{i \psi(x,\xi)}\, \dd \xi.
\end{equation}
Now, it follows from \ref{lem:fuijiwara} that for any  $\alpha\in (0,1)$, there exists a constant $c$ such that
\[
	\abs{K(x,z)}\leq c (1+\abs{z})^{-n-\alpha},
\]
and therefore $\sup_{x} \int | K(x,\mathbf{t}(x)-y) |\,dy <\infty.$ This yields at once the boundedness of the operator $T_\eta$ on $L^{\infty}.$  Moreover using the change of variables $z=\mathbf{t}(x),$ we observe that the determinant of its Jacobian, denoted by $\det \,J(z),$ is bounded from above by $\frac{1}{c},$ because of the SND condition $\abs{\det\, D \mathbf{t}(x)}\geq c>0.$ Therefore
\[
    \begin{split}
	\sup_{y}\int  | K(x,\mathbf{t}(x)-y) |\, \dd x &=\sup_{y}\int |K(\mathbf{t}^{-1}(z),z-y) |\det\,J(z)|\, \dd z  \\
        &\leq \frac{1}{c} \sup_{y}\int (1+\abs{z-y})^{-n-\alpha}\, dz < \infty,
    \end{split}
\]
where we have also used \eqref{low frequency kernel estim}. Therefore Schur's lemma yields that $T_\eta$ is bounded on $L^q$ for all $q\in [1,\infty]$ and this ends the proof of the theorem.
\end{proof}

\subsection{Boundedness of the high frequency part of the FIO}\label{subsec high freq boundedness}
The Seeger-Sogge-Stein decomposition, yields a decomposition of the FIO into low and high frequency parts. In subsection \ref{subsec low freq} we established the $L^q -L^r$ boundedness of linear low frequency rough FIOs and therefore the remaining part of the boundedness problem is the treatment of the high frequency part.

\begin{thm}\label{thm:Linear}
 Suppose that $0<r\leq \infty$, $1\leq p,q \leq \infty$, satisfy the relation $1/r = 1/q + 1/p$. Let $a \in L^p S^m_\varrho$, $\varphi\in \Phi^2$ satisfying the SND condition, $\varrho \leq 1$, $s:= \min(2,p,q),$ $\frac{1}{s}+\frac{1}{s'}=1$ and
\begin{equation*}
	m < -\frac{(n-1)}{2}\brkt{\frac{1}{s}+\frac{1}{\min(p,s')}} + \frac{n(\varrho-1)}{s}.
\end{equation*}
Then the operator $T_a$ is bounded from $L^q$ to $L^r$ and its norm is bounded by a constant $C$, depending only on $n$, $m$, $\varrho$, $p$, $q$, and a finite number of $C_\alpha$'s in Definition $\ref{LpSmrho definition}.$
\end{thm}

\begin{proof}
We shall assume that $q<\infty$. The case $q=\infty$ is proved with minor modifications in the argument, so we omit the details. We would like to prove that there exists a constant $C$, depending only on $n$, $m$, $\varrho$, $p$, $q$ and a finite number of $C_\alpha$'s in Definition \ref{LpSmrho definition}, such that
\[
\|T_a u\|_{L^r (\R^n)} \leq C\|u\|_{L^q (\R^n)},
\]
for all $u\in \S$.
To achieve this, we decompose $T_a$ as in \eqref{Tdecomp}. By Theorem \ref{Linearlow}, the first term $T_0$, satisfies the desired boundedness, so as mentioned above, we confine ourselves to the analysis of the second term $ \sum_{j=1}^{\infty}\sum_{\nu}T^{\nu}_{j} u(x)$ in
\eqref{Tdecomp}. Here we use the representation \eqref{defn rewrittentnuj} of the operators $T^{\nu}_{j}$ namely
\begin{equation*}
 T^{\nu}_{j} u(x)
=\frac{1}{(2\pi)^{n}}\int_{\mathbb{R}^{n}}A^{\nu}_{j}(x,\xi) e^{i
\langle (\nabla_{\xi}\varphi)(x,\xi^{\nu}_{j}),\, \xi\rangle} \hat{u}(\xi)\, \dd\xi.
\end{equation*}
This can be rewritten as
\begin{equation*}
 T^{\nu}_{j} u(x)
=\int_{\mathbb{R}^{n}} K^{\nu}_{j} (x,(\nabla_{\xi}\varphi)(x,\xi^{\nu}_{j})-y) u(y) dy
\end{equation*}
with
\begin{equation*}
K^{\nu}_{j} (x,z)
=\frac{1}{(2\pi)^{n}}\int_{\mathbb{R}^{n}}A^{\nu}_{j}(x,\xi) e^{i
\langle z,\, \xi\rangle}\, \dd\xi.
\end{equation*}
Let $L$ be the differential operator given by
\begin{equation*}
L=I-2^{2j}\frac{\partial ^{2}}{\partial \xi_1 ^{2}}-2^{j}
\Delta_{\xi'}.
\end{equation*}
Using the definition of $A^{\nu}_{j} (x,\xi)$ in \eqref{amplitude}, the assumption that $a\in L^{p}S^{m}_{\varrho}$ together with \eqref{chiestimate 1}, \eqref{chiestimate 2}, and the uniform estimates (in $x$) for $\Phi(x,\xi)$ in \eqref{phaseestim1} and \eqref{phaseestim2}, we can show that for any $\nu$ and any $\xi\in \sup_\xi {A^\nu_j}$
\begin{equation*}
\Vert L^N A^{\nu}_{j}(\cdot,\xi)\Vert_{L^{p}}\leq C_{N} 2^{j(m+ 2N(1-\varrho))}.
\end{equation*}
Let ${\mathbf t}_j^\nu(x)=(\nabla_{\xi}\varphi)(x,\xi^{\nu}_{j})$ and $\alpha\in (0,\infty)$. As before, the SND condition on the phase function yields that $|\det D {\mathbf t}_j^\nu(x)|\geq c>0.$ Setting
\begin{equation*}
g(y):= (2^{2j} y^2_1 + 2^{j} |y'|^2)^{\frac{\alpha}{2}}
\end{equation*}
we can split
\[
\begin{split}
\textbf{I}_1 + \textbf{I}_2&:=\sum_{\nu}\left(\int_{g(y)\leq 2^{-j\varrho}} +\int_{g(y)> 2^{-j\varrho}}\right)\vert K_{j}^{\nu}(x,y)u({\mathbf t}_j^\nu(x)-y) \vert\, \dd y\\
&=\sum_{\nu}\int\vert K_{j}^{\nu}(x,y)u({\mathbf t}_j^\nu(x)-y)\vert\, \dd y.
\end{split}
\]
H\"older's inequality in $\nu$ and $y$ simultaneously and thereafter, since $1\leq s\leq 2$, the Hausdorff-Young inequality in the $y$ variable of the second integral yields
\begin{equation*}
\begin{split}
	\textbf{I}_1 &\leq \{\sum_{\nu}\int_{g(y)\leq 2^{-j\varrho}} \abs{u({\mathbf t}_j^\nu(x)-y)}^s \dd y\}^{\frac{1}{s}}\{\sum_{\nu}\int\vert K_{j}^{\nu}(x,y)\vert^{s'} \dd y\}^{\frac{1}{s'}}\\
&\lesssim  \{\sum_{\nu}\int_{g(y)\leq 2^{-j\varrho}} \abs{u({\mathbf t}_j^\nu(x)-y)}^s \dd y\}^{\frac{1}{s}}\{\sum_{\nu}\brkt{\int\vert A_{j}^{\nu}(x,\xi)\vert^{s}\, \dd\xi}^{\frac{s'}{s}}\}^{\frac{1}{s'}}.
\end{split}
\end{equation*}
If we now set $U_{j}^\nu(x,y):=u({\mathbf t}_j^\nu(x)-y),$ raise the expression in the estimate of $\textbf{I}_1$ to the $r$-th power and integrate in $x$, then H\"older's inequality yields that $\norm{\textbf{I}_1}_{L^r}$ is bounded a constant times
\begin{equation}\label{I}
\set{\int \brkt{\sum_{\nu}\int_{g(y)\leq 2^{-j\varrho}} \abs{U^{\nu}_j(x,y)}^s \dd y}^{\frac{q}{s}} \dd x}^{\frac{1}{q}}
\set{ \int \brkt{\sum_{\nu}\brkt{\int\vert A_{j}^{\nu}(x,\xi)\vert^{s}\, \dd\xi}^{\frac{s'}{s}}}^{\frac{p}{s'}} \dd x}^{\frac{1}{p}}.
\end{equation}
We shall deal with the two terms in the right hand side of this estimate separately. To this end using the Minkowski integral inequality (simultaneously in $y$ and $\nu$), we can see that the first term is bounded by
\[\set{\sum_{\nu}\int_{g(y)\leq 2^{-j\varrho}}  \brkt{\int \abs{U^{\nu}_j(x,y)}^q \dd x}^{\frac{s}{q}} \dd y}^{\frac{1}{s}}.\]
Observe now that, letting ${\mathbf t}_j^\nu(x)=t$ and using $\abs{\det D\,{\mathbf t}_j^\nu(x)}\geq c>0$, we get
\begin{equation}\label{eq:u}
	\brkt{\int \abs{U^{\nu}_j(x,y)}^q \dd x}^{\frac{1}{q}}=\brkt{\int \abs{u(t-y)}^q \abs{\det D\,{\mathbf t}_j^\nu(x)}^{-1} \dd t}^{\frac{1}{q}}\leq c^{-\frac{1}{q}} \norm{u}_{L^q}.
\end{equation}
Thus, the first term on the right hand side of \eqref{I} is bounded by a constant multiple of
\begin{equation} \label{first term of I}
\begin{split}
	\set{\sum_\nu \int_{g(y)\leq 2^{-j\varrho}} \dd y}^{\frac{1}{s}}  \norm{u}_{L^q}&\lesssim
	2^{j\frac{n-1}{2s}} 2^{-j\frac{n+1}{2s}}  \set{\int_{\abs{y}\leq 2^{-j\frac{\varrho}{\alpha}}} \dd y}^{\frac{1}{s}} c \norm{u}_{L^q}\\&\lesssim 2^{j\frac{n-1}{2s}} 2^{-j\frac{n+1}{2s}}  2^{-j\frac{n}{\alpha s}} \norm{u}_{L^q}.
\end{split}
\end{equation}
To analyse the second term we shall consider two separate cases, so assume first that $p\geq s'.$ Minkowski inequality yields that the second term in the right hand side of \eqref{I} is bounded by
\begin{equation*}
\begin{split}
&\set{\sum_{\nu} \sqrbrak{\int \brkt{\int |A^{\nu}_{j} (x,\xi)|^{s}\dd \xi}^{\frac{p}{s}}\dd x}^{\frac{s'}{p}}}^{\frac{1}{s'}}\leq  \set{\sum_{\nu} \sqrbrak{\int \brkt{\int |A^{\nu}_{j} (x,\xi)|^{p}\dd x}^{\frac{s}{p}}\dd \xi}^{\frac{s'}{s}}}^{\frac{1}{s'}}\\ &\lesssim 2^{jm} \brkt{\sum_{\nu} |\supp_{\xi} A^{\nu}_{j}|^{\frac{s'}{s}}}^{\frac{1}{s'}}\lesssim 2^{jm} 2^{j\frac{n+1}{2s}}2^{j\frac{n-1}{2s'}},
\end{split}
\end{equation*}
where we have used the fact that the measure of the $\xi-$support of $A^{\nu}_{j}$ is $O(2^{j\frac{n+1}{2}}).$ Now let $p<s'$, then the second term on the right hand side of \eqref{I} is bounded by
\begin{equation*}
\begin{split}
\set{\sum_{\nu} {\int \brkt{\int |A^{\nu}_{j} (x,\xi)|^{s}\dd \xi}^{\frac{p}{s}}\dd x}}^{\frac{1}{p}}&\leq  \set{\sum_{\nu} \sqrbrak{\int \brkt{\int |A^{\nu}_{j} (x,\xi)|^{p}\dd x}^{\frac{s}{p}}\dd \xi}^{\frac{p}{s}}}^{\frac{1}{p}}\\ &\lesssim 2^{jm} \brkt{\sum_{\nu} |\supp_{\xi} A^{\nu}_{j}|^{\frac{p}{s}}}^{\frac{1}{p}}\lesssim 2^{jm} 2^{j\frac{n+1}{2s}}2^{j\frac{n-1}{2p}}.
\end{split}
\end{equation*}
Therefore using \eqref{first term of I} and the estimates for the second term on the right hand side of \eqref{I}, we obtain
\[
	\norm{\textbf{I}_1}_{L^r} \lesssim 2^{j\brkt{m-\varrho \frac{n}{\alpha s}+\frac{n-1}{2}\brkt{\frac{1}{s}+\frac{1}{\min(p,s')}}}}\norm{u}_{L^q},
\]
and the constant hidden on the right hand side of this estimate does not depend on $\alpha$.

Define $h(y)=1+2^{2j} y^2_1 + 2^{j} |y'|^2$ and let $M>\frac{n}{2s}$. By H\"older's inequality,
\begin{multline}
	\norm{\textbf{I}_2}_{L^r}\leq\set{\int\{\sum_\nu\int_{g(y)>  2^{-j\varrho}}\abs{U^{\nu}_j(x,y)}^sh(y)^{-sl}\, \dd y\}^{\frac{q}{s}}\, \dd x}^{\frac{1}{q}}\\
	\times \set{\int \{\sum_\nu\int\vert {K_{j}^{\nu}}(x,y)\,h(y)^{M}\vert^{s'}\, \dd y\}^{\frac{p}{s'}}\, \dd x}^{\frac{1}{p}}.
\end{multline}
By Minkowski's integral inequality and \eqref{eq:u}, the first term of the right hand side is bounded by
a constant times \begin{equation}\label{eq:estimate1}
	\begin{split}
		\norm{u}_{L^q} \set{\sum_\nu\int_{g(y)>  2^{-j\varrho}}h(y)^{-sl}\, \dd y}^{\frac{1}{s}} &\lesssim \norm{u}_{L^q}2^{j\frac{n-1}{2s}}2^{\frac{-j(n+1)}{2s}}\{\int_{|y|> 2^{-j\frac{\varrho}{\alpha}}}|y|^{-2sl}\, \dd y\}^{\frac{1}{s}} \\
		&\lesssim \norm{u}_{L^q }2^{j\frac{n-1}{2s}} 2^{\frac{-j(n+1)}{2s}}  2^{j\frac{\varrho}{\alpha} (2M-\frac{n}{s})}.
	\end{split}
\end{equation}
In order to control the second term, let us assume first that $M\in \Z_+$. In this case, Hausdorff-Young's inequality, Minkowski's integral inequality, and the same argument as in the analysis of $\textbf{I}_1$ yield
\begin{multline}\label{integer M estimates for I2}
	\set{\int \{\sum_\nu\int\vert {K_{j}^{\nu}}(x,y)\,h(y)^{M}\vert^{s'}\, \dd y\}^{\frac{p}{s'}}\, \dd x}^{\frac{1}{p}}\leq \\\leq
	\set{\int \{\sum_\nu\brkt{\int\abs{L^M A_{j}^{\nu}(x,\xi)}^s\, \dd \xi}^{\frac{s'}{s}}\}^{\frac{p}{s'}}\, \dd x}^{\frac{1}{p}}
	\lesssim 2^{j(m+ 2M(1-\varrho))} 2^{j\frac{n+1}{2s}} 2^{j\frac{n-1}{2\min(s',p)}}.
\end{multline}
Assume now that $M$ is not an integer. Then we can write it as $[M]+\{M\}$ where $[M]$ denotes the integer part of $M$ and $\{M\}$ its fractional part, which is in the interval $(0,1)$. Therefore, H\"older's inequality with conjugate exponents $\frac{1}{\{M\}}$ and $\frac{1}{1-\{M\}}$ yields

\[
    \begin{split}
        \sum_\nu \int\vert &{K_{j}^{\nu}}(x,y)\,h(y)^{M}\vert^{s'}\,\dd y\\
        &= \sum_\nu \int\vert {K_{j}^{\nu}}(x,y)\vert^{s'\{M\}}\, \vert {K_{j}^{\nu}}(x,y)\vert^{s'(1-\{M\})}h(y)^{s'\{M\}([M]+1)}\, h(y)^{s'[M](1-\{M\})} \,\dd y \\
            &\leq \bigg(\sum_\nu \int \vert {K_{j}^{\nu}}(x,y)\vert^{s'} h(y)^{s'([M]+1)} \, \dd y\bigg)^{\{M\}} \bigg(\sum_\nu \int \vert {K_{j}^{\nu}}(x,y)\vert^{s'}h(y)^{s'[M]}\, \dd y\bigg)^{1-\{M\}}.
\end{split}
\]
Thus,
\begin{multline*}
	{\int \brkt{\sum_\nu \int\vert {K_{j}^{\nu}}(x,y)\,h(y)^{M}\vert^{s'}\, \dd y}^{\frac{p}{s'}}\, \dd x}
	\leq
	\set{\int \bigg(\sum_\nu \int \vert {K_{j}^{\nu}}(x,y)\vert^{s'} h(y)^{s'([M]+1)} \, \dd y\bigg)^{\frac{p}{s'}}\, \dd x}^{\{M\}}\\
	\times \set{\int \bigg(\sum_\nu \int \vert {K_{j}^{\nu}}(x,y)\vert^{s'}h(y)^{s'[M]}\, \dd y\bigg)^{\frac{p}{s'}}\, \dd x}^{1-\{M\}}.
\end{multline*}
Therefore, using \eqref{integer M estimates for I2} we obtain
\begin{equation}\label{eq:kernelestimate}
    \set{\int \{\sum_\nu \int\vert {K_{j}^{\nu}}(x,y)\,h(y)^{M}\vert^{s'}\, \dd y\}^{\frac{p}{s'}}\, \dd x}^{\frac{1}{p}}\leq C_{N} 2^{j(m+ 2M(1-\varrho))} 2^{j\frac{n+1}{2s}}2^{j\frac{n-1}{2\min(s',p)}}.
\end{equation}
Hence, for every $2M>\frac{n}{s}$, \eqref{eq:kernelestimate} and \eqref{eq:estimate1} yields
\[
	\norm{\textbf{I}_2}_{L^r} \lesssim 2^{j(m+2M(1-\varrho))}2^{j\frac{\varrho}{\alpha} (2M-\frac{n}{s})} 2^{j\frac{n-1}{2}\brkt{\frac{1}{s}+\frac{1}{\min(p,s')}}}\norm{u}_{L^q},
\]
with a constant independent of $\alpha$.
Now putting the estimates for $\textbf{I}_1$ and $\textbf{I}_2$ together and summing, yield that for any $\alpha>0$,
\[
    \norm{T_j u}_{L^r}\lesssim \brkt{2^{j(m+2M(1-\varrho)+\frac{n-1}{2}\brkt{\frac{1}{s}+\frac{1}{\min(p,s')}})}2^{j\frac{\varrho}{\alpha} (2M-\frac{n}{s})}+2^{j\brkt{m-\varrho \frac{n}{\alpha s}+\frac{n-1}{2}\brkt{\frac{1}{s}+\frac{1}{\min(p,s')}}}}}\norm{u}_{L^q},
\]
Therefore letting $\alpha$ tend to $\infty,$ we obtain
\[
    \norm{T_j u}_{L^r}\lesssim 2^{j(m+2M(1-\varrho)+\frac{n-1}{2}\brkt{\frac{1}{s}+\frac{1}{\min(p,s')}})}\norm{u}_{L^q}.
\]
Now if we let $R:=\min(r,1),$ we get
\[
	\norm{\sum_{j=1}^{\infty}T_{j} u}_{L^r}^R\leq \sum_{j=1}^{\infty}\norm{T_{j} u}_{L^r}^R\lesssim
	\sum_{j=1}^\infty 2^{jR \brkt{\frac{n-1}{2}\brkt{\frac{1}{s}+\frac{1}{\min(p,s')}}+m+2M(1-\varrho)}}\norm{u}_{L^q}^R\lesssim \norm{u}_{L^q}^R,
\]
provided  $m< -\frac{n-1}{2}\brkt{\frac{1}{s}+\frac{1}{\min(p,s')}}+2M(\varrho-1)$, for $2M>\frac{n}{s}$.
Therefore, the inequality holds provided
\[
    m< -\frac{n-1}{2}\brkt{\frac{1}{s}+\frac{1}{\min(p,s')}}+\frac{n(\varrho-1)}{s}.
\]
\end{proof}

In the case that $T_a$ is a pseudodifferential operator, with minor modifications in the previous argument we obtain the following result, which improves \cite[Thm. 5.2]{MRS} for the case $w=1$ and $\mu$ the Lebesgue measure:
\begin{thm}\label{thm:pseudodifferential}
Let $a \in L^p S^m_\varrho$, $\varphi\brkt{x,\xi}=\pr{x}{\xi}$ and suppose that $0<r\leq \infty$, $1\leq p,q \leq \infty$, satisfy the relation $1/r = 1/q + 1/p$. Suppose further that $\varrho \leq 1$, $s=\min(2,p,q)$ and
\[
	m < \frac{n(\varrho-1)}{s}.
\]
Then the operator $T_a$ is bounded from $L^q$ to $L^r$ and its norm is bounded by a constant $C$, depending only on $n$, $m$, $\varrho$, $p$, $q$, and a finite number of $C_\alpha$'s in Definition $\ref{LpSmrho definition}.$
\end{thm}

In the case of $q=2\leq p,$ Theorem \ref{thm:Linear} can be improved to yield a result similar to Theorem \ref{thm:pseudodifferential} for FIOs, but before we proceed to that, we will need a couple of lemmas.

\begin{lem}\label{lem:calculus_lemma} If $a\in L^pS^{m_1}_\varrho$ and $b\in L^q S^{m_2}_\varrho$ then $a\cdot b\in L^r S^{m_1+m_2}_{\varrho}$ where $\frac1 r =\frac1 p+\frac1 q$, $1\leq p,q\leq \infty$.
Moreover, if $\eta(\xi)\in C^{\infty}_{0}$ and $a_\varepsilon(x,\xi):=a(x,\xi)\eta(\varepsilon \xi)$ and $\varepsilon \in[0,1),$ then one has
\[
    \sup_{0<\varepsilon\leq 1}\sup_{\xi\in \R^n} \langle \xi\rangle^{-m+\varrho |
\alpha|} \norm{\d^\alpha_\xi a_\varepsilon (\cdot,\xi)}_{L^p}\leq c_{\eta,\abs{\alpha},\varrho}\abs{a}_{p,m,\abs{\alpha}}.
\]

\end{lem}
\begin{proof} The result follows directly from Leibniz's rule and H\"older's inequality.
\end{proof}

\begin{lem}\label{lem:H_M} Let $2\leq p\leq \infty$, $0\leq \varrho\leq 1$, $a \in L^p S^m_\varrho$ and $r=\frac{2p}{p+2}.$ For $u\in \mathscr{S},$ a real number $M>n,$ and all multi-indices $\alpha,\beta$ with $\beta\leq \alpha$, set
\begin{equation}
\label{eq:H_M}
	H_M^{\alpha,\beta} u(x,\xi):= |\partial^{\beta} a(x,\xi)| {\int \brkt{1+2^j\abs{x-y}}^{-M} | \partial^{\alpha-\beta} a(y,\xi)| |u(y)|\, \dd y }.
\end{equation}
Then for every $u\in L^{r'}$
\[
	\norm{H_M^{\alpha,\beta}u(\cdot,\xi)}_{L^{r}}\leq C_{M}
	\abs{a}_{p,m,\abs{\alpha}}^2  2^{-jn} \left< \xi\right>^{2m-\varrho\abs{\alpha}}\norm{u}_{L^{r'}}.
\]
\end{lem}
\begin{proof}
Since $\frac{1}{r}=\frac{1}{p}+\frac{1}{2}$, H\"older's and Minkowski's inequalities yield
\[
	\begin{split}
	\norm{H_M^{\alpha,\beta}u(\cdot,\xi)}_{L^{r}}&\leq \Vert{\partial^{\beta} a(\cdot,\xi)}\Vert_{L^p}
	 \norm{{\int  \brkt{1+2^j\abs{y}}^{-M} |\partial^{\alpha-\beta} a(\cdot-y,\xi) u(\cdot-y)|\, \dd y }}_{L^2}\\
	 &\leq \Vert{\partial^{\beta} a(\cdot,\xi)}\Vert_{L^p} \int  \brkt{1+2^j\abs{y}}^{-M}\, \dd y \Vert u{\partial^{\alpha-\beta} a(\cdot,\xi) }\Vert_{L^2}\\
	 &\leq C_{M} 2^{-jn}\Vert{\partial^{\beta} a(\cdot,\xi)}\Vert_{L^p} \Vert u{\partial^{\alpha-\beta} a(\cdot,\xi) }\Vert_{L^2},
	\end{split}
\]
provided $M>n$.
On the other hand, since $\frac{1}{2}=\frac{1}{p}+\frac{1}{r'}$,
H\"older's inequality yields
\[
	\Vert u{\partial^{\alpha-\beta} a(\cdot,\xi) }\Vert_{L^2}\leq \Vert u\Vert_{L^{r'}} \Vert \partial^{\alpha-\beta} a(\cdot,\xi)\Vert_{L^p}.
\]
Therefore, since $a\in L^p S^m_\varrho$ one has
\[
	\norm{H_M^{\alpha,\beta}(\cdot,\xi)}_{L^{r}}\leq C_{M} \abs{a}_{p,m,\abs{\alpha-\beta}} \abs{a}_{p,m,\abs{\beta}}2^{-jn}\left< \xi\right>^{2m-\varrho\abs{\alpha}} \norm{u}_{L^{r'}},
\]
from which the result follows.
\end{proof}

\begin{thm}\label{thmL2}
Assume that $a \in L^p S^m_\varrho$ with $0\leq \varrho\leq 1$, $2\leq p\leq \infty$, and $\varphi\in \Phi^2$ is a phase function satisfying the SND condition. If $r=\frac{2p}{p+2}$ and $m < \frac{n(\varrho-1)}{2},$ then the operator $T_a$ is bounded from $L^2$ to $L^r$ and its norm is bounded by a constant $C$, depending only on $n$, $m$, $\varrho$, $p$, and a finite number of $C_\alpha$'s in Definition $\ref{LpSmrho definition}.$
\end{thm}

\begin{proof}
We define a Littlewood-Paley partition of unity as in \eqref{eq:LittlewoodPaley}. Set $a_j(x,\xi)=a(x,\xi)\Psi_j(\xi)$ for for $j\geq 0$.

By Lemma \ref{lem:calculus_lemma}, $a_j\in L^{p}S^m_\varrho$ and for any $s\in \Z^+$
\[
    \sup_{j\geq 0} \abs{a_j}_{p,m,s}\lesssim \abs{a}_{p,m,s}.
\]

That $T_{a_0}$ satisfies the required bound follows from Theorem \ref{Linearlow}, so it is enough to consider the boundedness of the operators $T_{a_j}$ for $j\geq 1$. To this end, we begin by studying the boundedness of $S_j:=T_{a_j}T^{\ast}_{a_j}$.  A simple calculation yields that $S_ju(x)=\int K_j(x,y) u(y)\, \dd y$ with
 \begin{equation*}
 K_j(x,y)=\frac{1}{(2\pi)^{n}}\int e^{i(\varphi(x,\xi)-\varphi(y,\xi))} a_j(x,\xi) \overline{a_j(y,\xi)}\, \dd\xi.
 \end{equation*}
Now since $\varphi$ is homogeneous of degree $1$ in the $\xi$ variable, the kernel $K_j (x,y)$ can be written as
\begin{equation*}
K_{j}(x,y)=\frac{2^{jn}}{(2\pi)^{n}}\int m_{j}(x,y,2^{j}\xi)
e^{i2^j \Phi(x,y,\xi)}\, \dd\xi.
\end{equation*}
with $\Phi(x,y,\xi)= \varphi (x,\xi) -\varphi (y,\xi)$ and $m_j(x,y,\xi)=a_j(x,\xi) \overline{a_j(y,\xi)}$. Observe that the support of $m_{j}(x,y,2^{j}\xi)$ lies in the compact set $\mathcal{K}=\{\frac{1}{2}\leq \abs{\xi}\leq 2\}$.  From the mean value theorem, \eqref{C_alpha} and \eqref{lower_bound on mixed hessian}, it follows that
\begin{equation}
\label{Phi cond 1}
\vert \nabla_{\xi}\Phi (x,y, \xi)\vert \thickapprox \vert x-y\vert,
\end{equation}
for any $x,y\in \R^n$ and $\xi\in \mathcal{K}$.

We claim that, for any $M>n$ there is a constant $C_M$ depending only on $M$ such that
\begin{equation}\label{LPpiece}
	\norm{S_j u}_{L^{r}}\leq C_M \abs{a}_{p,m,[M]+1}^2 2^{2j m} 2^{jM(1-\varrho)} \norm{u}_{L^{r'}},
\end{equation}
for any $u\in L^{r'}$, where $[M]$ stands for the integer part of $M$.

Assume first that $M>n$ is an integer.  Fix $x\neq y$ and set $f(\xi):=\Phi (x,y, \xi)$, $\Psi=\abs{\nabla_\xi f}^2$. By the mean value theorem, \eqref{C_alpha} and \eqref{Phi cond 1}, for any multi-index $\alpha$ with $\abs{\alpha}\geq 1$ and any $\xi\in \mathcal{K}$,
\[
    \abs{\d^\alpha_{\xi} f(\xi)}\leq |\d^\alpha_\xi \nabla_x \varphi(z_{x,y},\xi)| |x-y|\lesssim \vert \nabla_{\xi}\Phi (x,y, \xi)\vert=\vert \nabla_{\xi}f\vert = \Psi^{1/2}.
\]
On the other hand, since
\[
    \d^\alpha \Psi=\sum_{j=1}^n \sum  \binom{\alpha}{\beta} \d^\beta\d_jf \d^{\alpha-\beta}\d_j f,
\]
it follows that, for any $\abs{\alpha}\geq 0$, $\abs{\d^\alpha \Psi}\lesssim \Psi$ and the constants are uniform on $x$ and $y$. Thus \eqref{Phi cond 1} and Lemma \ref{lem:technic} with $u=m_{j}(x,y,2^{j}\xi),$ $f=\Phi (x,y, \xi)$ yield
\begin{equation*}
\begin{split}
\vert K_{j}(x,y)\vert & \leq 2^{j n} 2^{-j M}\ C_{M,\mathcal{K}} \sum _{\vert \alpha \vert\leq M} 2^{j\abs{\alpha}}\int {\vert \partial^{\alpha}_{\xi} m_{j}(x,y,2^j \xi)\vert \vert \nabla_{\xi}\Phi(x,y,\xi)\vert^{-M}} \, \dd \xi\\
&\lesssim 2^{-j M} \abs{x-y}^{-M} \sum _{\vert \alpha \vert\leq M}  2^{j\abs{\alpha}} \int \abs{\partial^{\alpha}_{\xi} m_{j}(x,y,\xi)}\, \dd \xi.
\end{split}
\end{equation*}
On the other hand
\begin{equation*}
\vert K_{j}(x,y)\vert \leq \int \abs{m_j(x,y,\xi)}\, \dd \xi\lesssim \sum _{\vert \alpha \vert\leq M}  2^{j\abs{\alpha}} \int \abs{\partial^{\alpha}_{\xi} m_{j}(x,y,\xi)}\, \dd \xi.
\end{equation*}
Therefore
\begin{equation}\label{eq:Kk}
	\vert K_{j}(x,y)\vert\lesssim \brkt{1+2^j\abs{x-y}}^{-M}\sum _{\vert \alpha \vert\leq M}  2^{j\abs{\alpha}} \int \abs{\partial^{\alpha}_{\xi} m_{j}(x,y,\xi)}\, \dd \xi .
\end{equation}
Now since
\begin{equation}\label{eq:mk}
	\abs{\partial^{\alpha}_{\xi} m_{j}(x,y,\xi)}
	\leq \sum_\beta \binom{\alpha}{\beta}  \abs{\partial^{\beta} a_j(x,\xi)  \partial^{\alpha-\beta} a_j(y,\xi) },
\end{equation}
we obtain that
\begin{equation}\label{eq:S_K}
 	S_j u(x)\leq \sum_{\vert \alpha \vert\leq M} \sum_\beta  \binom{\alpha}{\beta}   2^{j\abs{\alpha}}  \int_{\abs{\xi}\sim 2^j}  H_M^{\alpha,\beta}u(x,\xi)\, \dd \xi,
\end{equation}
where, $H_M^{\alpha,\beta}$ is defined as in \eqref{eq:H_M}. Hence Minkowski's inequality, Lemma \ref{lem:H_M} and \eqref{eq:S_K} yield
\begin{equation}\label{eq:S_K_integers}
	\begin{split}
	\norm{S_j u}_{L^{r}}&\leq  c_M  \sum_{\vert \alpha \vert\leq M} \sum  \binom{\alpha}{\beta} \abs{a}_{p,m,\abs{\alpha}}^2 2^{j\abs{\alpha}}
	2^{-jn}  \norm{{u}}_{L^{r'}}\int_{\abs{\xi}\sim 2^j} \left<\xi \right>^{2m-\varrho \abs{\alpha}}\, \dd \xi \\
	& \leq c_M \abs{a}_{p,m,M}^2 2^{2jm} \norm{u}_{L^{r'}}\sum_{\vert \alpha \vert\leq M} 2^{\abs{\alpha}} 2^{j\abs{\alpha}(1-\varrho)}\\
	& \leq c_M \abs{a}_{p,m,M}^2 2^{2jm} 2^{jM(1-\varrho)} \norm{u}_{L^{r'}}.
	\end{split}
\end{equation}

Assume now that $M\geq n+1$ is a real number. Writing $M=[M]+\{M\}$ as the sum of its integer and fractional parts, the estimate \eqref{eq:S_K_integers} yields
\[
	\begin{split}
	\norm{S_j u}_{L^{r}} &=\norm{S_j u}_{L^{r}}^{1-\{M\}}\norm{S_j u}_{L^{r}}^{\{M\}}\\
&\leq
	\brkt{c_{[M]} \abs{a}_{p,m,[M]}^2 2^{2jm} 2^{j[M](1-\varrho)} \norm{u}_{L^{r'}}}^{1-\{M\}}\\
&\qquad\qquad\times\brkt{c_{[M]+1} \abs{a}_{p,m,[M]+1}^2  2^{2 j m} 2^{j([M]+1)(1-\varrho)} \norm{u}_{L^{r'}}}^{\{M\}}\\
	&\leq c_M  \abs{a}_{p,m,[M]+1}^2 2^{2km} 2^{kM(1-\varrho)} \norm{u}_{L^{r'}}.
	\end{split}
\]

Assume now that $n<M<n+1$. Then, writing $M=n+\{M\}$ and letting
\[
	R_{l}(x,y):=
	\sum _{\vert \alpha \vert\leq l}   \sum \binom{\alpha}{\beta}  2^{j\abs{\alpha}}\int_{\abs{\xi}\sim 2^{j}} \abs{\partial^{\beta} a(x,\xi)  \partial^{\alpha-\beta} a(y,\xi) }\, \dd \xi,
\]
we see that the application of \eqref{eq:Kk} and \eqref{eq:mk} with $n$ and $n+1$ yields
\[
	\begin{split}
	\abs{K_j(x,y)} &=\abs{K_j(x,y)}^{1-\{M\}}\abs{K_j(x,y)}^{\{M\}}\\
			&\leq {R_n(x,y)}^{1-\{M\}}  {R_{n+1}(x,y)}^{\{M\}}  \brkt{1+2^j\abs{x-y}}^{-M}.
	\end{split}
\]
Hence, applying H\"older's inequality with the exponents $\frac{1}{\{M\}}$ and $\frac{1}{1-\{M\}}$ we get
\begin{multline*}
	S_j u(x)\leq \brkt{\int R_n(x,y) \brkt{1+2^j\abs{x-y}}^{-M}\abs{u(y)}\, \dd y}^{1-\{M\}}\\
    \times \brkt{\int R_{n+1}(x,y) \brkt{1+2^j\abs{x-y}}^{-M}\abs{u(y)}\, \dd y}^{\{M\}},
\end{multline*}
and another application of the H\"older inequality with exponents $\frac{r}{\{M\}}$ and $\frac{r}{1-\{M\}}$ yields
\begin{multline*}
	\norm{S_j u}_{L^{r}}\leq \norm{\int R_n(x,y) \brkt{1+2^j\abs{x-y}}^{-M}\abs{u(y)}\, \dd y}^{1-\{M\}}_{L_{x}^{r}}\\
    \times
	\norm{\int R_{n+1}(x,y) \brkt{1+2^j\abs{x-y}}^{-M}\abs{u(y)}\, \dd y}^{\{M\}}_{L_{x}^{r}}.
\end{multline*}
Therefore, Minkowski's integral inequality and Lemma \eqref{lem:H_M} yield
\[
	\begin{split}
	\norm{S_j u}_{L^{r}}
	&\leq C_{M} \abs{a}_{p,m,n+1}^2 2^{2j m} 2^{j M(1-\varrho)} \norm{u}_{L^{r'}},
	\end{split}
\]
for all $u\in L^{r'}$.
Thus, using \eqref{LPpiece}, we obtain
\[
	\norm{T^*_{a_j} u}_{L^2}^2=\pr{u}{T_{a_j}T_{a_j}^*u}\leq \norm{u}_{L^{r'}} \norm{S_j u}_{L^{r}}\leq C_{M} \abs{a}_{p,m,[M]+1}^2 2^{2j m} 2^{j M(1-\varrho)}\norm{u}_{L^{r'}}^2,
\]
and so
\[
    \norm{T_{a_j} u}_{L^r}\leq C_{M} \abs{a}_{p,m,[M]+1} 2^{j m} 2^{j \frac{M(1-\varrho)}{2}}\norm{u}_{L^{2}},
\]
for every $u\in L^2$.

Now if $\varrho =1$ and $m<0$ we see that the sum of the Littlewood-Paley pieces $T_{a_j}$ converges and therefore $T_a$ is a bounded operator from $L^{2}$ to $L^{r}$. In case $0\leq \varrho<1$ then the condition $m<\frac{n}{2}(\varrho-1)$ implies that there is a $M_{0}$ with $n<M_{0}<\frac{-2m}{1-\varrho}$. So by choosing $M=M_0$, we have
 \begin{equation}
 \Vert T_{a_j}u\Vert_{L^{r}} \lesssim 2^{j m}2^{ j\frac{M_{0}(1-\varrho)}{2}}\Vert u\Vert_{L^{2}},
 \end{equation}
with $2m+M_{0}(1-\varrho)<0.$ This and the summation of the pieces yield the desired boundedness of $T_a$.

\end{proof}

Here, we shall define a couple of parameters which will appear as the order of our operators in the remainder of this paper.
\begin{defn}\label{defn rune M}

(a) Given $1\leq p,\, q\leq \infty$ define
\[
    \textarc{m}(\rho,p,q):=\left\{
      \begin{array}{ll}
         -\frac{(n-1)}{2}\brkt{\frac{1}{p}+\frac{1}{\min(p,q)}} + \frac{n(\varrho-1)}{\min(p,q)}, & \hbox{if $1\leq p<2$, or $p\geq 2$ and $1\leq q<p'$;} \\
        \frac{n(\varrho-1)}{2}-(n-1)\brkt{\frac{1}{2}-\frac{1}{q}}, & \hbox{if $2\leq p,q$;} \\
        \frac{n(\varrho-1)}{q}-\frac{(n-1)}{1-\frac{2}{p}}\brkt{\frac{1}{q}-\frac{1}{2}}, & \hbox{if $p> 2$ and $p'\leq q\leq 2$.}
      \end{array}
    \right.
\]
(b) Furthermore given $1<q<2$ we set
\[\mathcal{M}(\varrho,p,q):=\frac{n(\varrho-1)}{q}-\frac{n-1}{1+1/p}\brkt{\frac{1}{q}-\frac{1}{2}} \]

\end{defn}
Using the notion above we can prove the following theorem:
\begin{thm}\label{thm:p2} Suppose that $0\leq \varrho\leq 1$ and $p\geq 2$, $1\leq q\leq \infty$, $0<r\leq \infty$, satisfy $\frac{1}{r}=\frac{1}{p}+\frac{1}{q}$. Also, assume that $\varphi\in \Phi^2$ satisfies the SND condition and $a\in L^p S^m_\varrho$ with
$m<\textarc{m}(\varrho,p,q).$
Then the operator $T_a$ is bounded from $L^q$ to $L^r$ and its norm is bounded by a constant $C$, depending only on $n$, $m$, $\varrho$, $p$, $q$, and a finite number of $C_\alpha$'s in Definition $\ref{LpSmrho definition}.$
Furthermore, when $1<q<2$ and
\begin{equation}\label{eq:m_4}
\textarc{m}(\varrho,p,q)\leq m<\mathcal{M}(\varrho,p,q)
\end{equation}
$T_a$ is bounded from $L^q$ to the Lorentz space $L^{r,q}$ and its norm is bounded by a constant $C$ with the same properties as above.
\end{thm}
\begin{proof} This follows by interpolating the result of Theorem \ref{thmL2} with the extremal results of Theorem \ref{thm:Linear} using Riesz-Thorin and Marcinkiewicz interpolation theorems respectively.
\end{proof}

\section{Global boundedness of multilinear FIOs}\label{bilinear FIO}

In this section we shall apply the boundedness of the linear FIOs obtained in the previous section to the problem of boundedness of bilinear and multilinear operators.

\subsection{Boundedness of bilinear FIOs}

Using an iteration procedure, we are able to reduce the problem of global boundedness of bilinear FIOs to that of boundedness of rough and linear FIOs. Our main result in this context is as follows.

\begin{thm} Let $a\in L_\Pi ^{p}S^\m_\rho(n,2)$ with $\m=(m_1 , m_2)\in \R_{-} \times \R_{-},$ $\rho=(\varrho_1 , \varrho_2)\in [0,1]^{2},$ $1\leq p \leq \infty$, and $\varphi_1,\varphi_2\in \Phi^2$ satisfy the SND condition. Let $1\leq q_1,q_2\leq \infty$. Assume that $q_1=\max(q_1,q_2)\geq p'$. Let
\[
    \frac{1}{r}=\frac{1}{p}+\frac{1}{q_1}+\frac{1}{q_2},
\]
and assume that
\[
    m_1<\textarc{m}(\varrho_1,p,q_1) \quad {\rm and} \quad m_2<\textarc{m}(\varrho_2,r_2,q_2),
\]
with $\frac{1}{r_2}=\frac{1}{p}+\frac{1}{q_1}.$
Then the bilinear FIO $T_a$, defined by
\begin{equation}
\label{Intro:bilinear_Fourier integral operator}
 T_{a} (f,g)(x)= \iint a(x,\xi,\eta)\, e^{ i \varphi_{1}(x, \xi)+ i\varphi_{2}(x, \eta)} \hat{ f}(\xi) \hat{g}(\eta)\,d\xi\, d\eta
\end{equation}
satisfies the estimate
\[
    \norm{T_a(f,g)}_{L^r}\leq C_{a,n} \norm{f}_{L^{q_1}}\norm{g}_{L^{q_2}},
\]
for every $f,\,g\in \S$. Moreover, if $1\leq q_2<2\leq r_2$,
\[
   m_1<\textarc{m}(\varrho_1,p,q_1) \quad {\rm and} \quad \textarc{m}(\varrho_2,r_2,q_2)\leq m_2<\mathcal{M}(\varrho_2,r_2,q_2),
\]
then
\[
    \norm{T_a(f,g)}_{L^{r,q_2}}\leq C_{a,n} \norm{f}_{L^{q_1}}\norm{g}_{L^{q_2}},
\]
for every $f,\,g\in \S$.
\end{thm}
\begin{proof}
For any $f,\, g\in \S$ set
\[
    \widetilde{a}\brkt{x,\eta}:=\int e^{i\varphi_1(x,\xi)} a(x,\xi,\eta) \widehat{f}(\xi)\, \dd \xi.
\]
Observe that the amplitude $\partial^{\alpha}_{\eta}a(\cdot,\xi,\eta)\in L^pS^{m_1}_{\varrho_1}$ if $\eta $ is hold fixed, and moreover for any $s\in \Z_+$,
\[
    \abs{\partial^{\alpha}_{\eta}a(\cdot,\cdot,\eta)}_{m_1,p,s}\leq c_{\alpha,s}\p{\eta}^{m_2-\varrho_2 \abs{\alpha}}.
\]
Thus, depending on the range of indices, we apply Theorem \ref{thm:Linear} or Theorem \ref{thm:p2} to obtain
\[
    \norm{\partial^{\alpha}_{\eta}\widetilde{a}\brkt{\cdot,\eta}}_{L^{r_2}}\lesssim \p{\eta}^{m_2-\varrho_2 \abs{\alpha}} \norm{f}_{L^{q_1}},
\]
provided $m_1<\textarc{m}(\varrho_1,p,q_1)$. This means that $\widetilde{a}\in L^{r_2}S^{m_2}_{\varrho_2}$ and for any $s\in \Z_+$,
\[
    \abs{\widetilde{a}}_{r_2,m_2,s}\lesssim \norm{f}_{L^{q_1}}.
\]
Now applying either Theorem \ref{thm:Linear} or Theorem \ref{thm:p2} again, we obtain the desired result.
\end{proof}

\begin{cor}\label{cor: cor of the main bilinear} Assume that $a\in L^\infty S^{m}_\varrho(n,2)$ and $\varphi_1,\varphi_2\in \Phi^2$ satisfy the SND condition. For $1\leq q_1,q_2\leq \infty$ let
\[
    \frac{1}{r}=\frac{1}{q_1}+\frac{1}{q_2},
\]
and assume that $q_1=\max(q_1,q_2)$ and
\[
    m<\textarc{m}(\varrho,\infty,q_1)+\textarc{m}(\varrho,q_1,q_2).
\]
Then the bilinear FIO $T_a$ defined by \eqref{Intro:bilinear_Fourier integral operator} satisfies
\[
    \norm{T_a(f,g)}_{L^r}\leq C_{a,n} \norm{f}_{L^{q_1}}\norm{g}_{L^{q_2}},
\]
for every $f,\,g\in \S$. Moreover, if $1\leq q_2<2\leq q_1$ and
\[
  \textarc{m}(\varrho,\infty,q_1)+ \textarc{m}(\varrho,q_1,q_2)\leq m<\textarc{m}(\varrho,\infty,q_1)+\mathcal{M}(\varrho,q_1,q_2)
\]
then
\[
    \norm{T_a(f,g)}_{L^{r,q_2}}\leq C_{a,n} \norm{f}_{L^{q_1}}\norm{g}_{L^{q_2}},
\]
for every $f,\,g\in \S$.
\end{cor}
\begin{rem} In the case $\varrho=1,\, r=1$ the previous corollary yields a global bilinear $L^2\times L^2\to L^1$ extension of H\"ormander and Eskin's local $L^2$ boundedness of zeroth order linear FIOs. Observe that in our global case, it suffices that the order $m$ is strictly negative since $\textarc{m}(1,\infty,2)+\textarc{m}(1,2,2)=0$. Furthermore for $m<\textarc{m}(1,\infty,\infty)+\textarc{m}(1,\infty,2)=-\frac{n-1}{2}$ we get the $L^\infty \times L^2 \to L^2$ boundedness of $T_a$ (see Theorem \ref{Thm:generalmultilinear} for the non-endpoint case).

Somewhat more interestingly, in the case $p=q_1=\infty$, $q_2=1$ and $\varrho=1$, the $L^\infty\times L^1\to L^1$ boundedness is valid provided the order $m<-n+1$. This can be compared with the result in \cite{MR2679898} where a local result has been obtained for a class of FIOs with more general amplitudes and phases than ours but with $m<-n + \frac{1}{2}$.
\end{rem}

We will illustrate the previous result with an application concerning certain bilinear oscillatory integrals. For the sake of simplicity, we will consider only the case $q_1,q_2\geq 2$.
\begin{cor} Let $\varphi_1,\varphi_2\in \Phi^2$ satisfying the SND condition. Let $2\leq q_1,q_2\leq \infty$ and
\[
    \frac{1}{r}=\frac{1}{q_1}+\frac{1}{q_2}.
\]
Let
$
    \sigma(\xi_1,\xi_2)=\frac{e^{i\abs{\xi}^\alpha}}{\abs{\xi}^\beta} \theta\brkt{\xi}$,  with $\alpha\in (0,1), \beta>0$,
where $\xi=(\xi_1,\xi_2)\in \R^{2n}$ and $\theta$ is a smooth function on $\R^{2n}$, which vanishes near the origin and equals $1$ outside a bounded set. Assume that
\[
    \alpha\in (0,1),\qquad \beta>\alpha n+(n-1)\brkt{1-\frac{1}{r}},
\]
Define for $0<t\leq 1$, $\sigma_t(\xi)=t^{\beta}a(t\xi)$. Then
\[
    \norm{\sup_{0<t\leq 1} \abs{T_{\sigma_t}\brkt{f,g}}}_{L^r}\lesssim \norm{f} _{L^{q_1}}\norm{g}_{L^{q_2}},
\]
for all $f,\,g\in \S$.

\end{cor}
\begin{proof} In order to prove the result, it suffices to consider the bilinear FIO with amplitude
\[
    a(x,\xi_1,\xi_2)=\frac{e^{i\abs{t(x)\xi}^\alpha}}{\abs{\xi}^\beta} \theta\brkt{t(x)\xi},
\]
for an arbitrary measurable function $t(x)\in [0,1]$. It can be shown that $a\in L^\infty S^{-\beta}_{1-\alpha}(n,2)$ for the given range of $\alpha,\beta ,$ and thereby the result follows from Corollary \ref{cor: cor of the main bilinear}.
\end{proof}

\parindent 0pt
\section{Boundedness of multilinear FIOs}
The following theorem yields the boundedness of a rather large class of rough multilinear Fourier integral operators on $L^{r}$ spaces for $0<r\leq \infty$. In the case of operators defined with phase functions that are inhomogeneous in the $\xi$-variable, i.e. more general multilinear oscillatory integral operators, we are also able to show a boundedness result in case the multilinear operator acts on $L^2$ functions. More precisely we have

\begin{thm}\label{Thm:generalmultilinear} Let $1\leq p\leq \infty$, $m_{j}<0,$ $j= 1,\dots N,$ and suppose that $\frac{\sum_{j=1}^{N} m_{j}}{\min_{j=1,\dots,N}m_{j}}\geq \frac{2}{p}$. Assume that the amplitude $a(x,\xi_{1},\dots, \xi_{N})\in L^{p}_{\Pi} S^{(m_1, \dots, m_N)}_{(1,\dots, 1)}(n,N)$ and the phase functions $\varphi_j\in \Phi^2,$ $j=1, \dots, N,$ are all strongly non-degenerate and belong to the class $\Phi^2.$

For $1\leq q_j < \infty$ in case $p=\infty,$ and $1\leq q_j \leq \infty$ in case $p\neq \infty,$ $j=1,\dots,N, $ let
\[
    \frac{1}{r}=\frac{1}{p}+\sum_{j=1}^{N}\frac{1}{q_j}.
\]
Then the multilinear FIO $T_a$, given by \eqref{defn RS FIO} or its equivalent representation

\[T_a(f_1, \dots, f_N) (x)= \int_{\R^{Nn}} a(x,\xi_1, \dots, \xi_N)\, e^{ i\sum_{j=1}^{N} \varphi_{j}(x, \xi_{j})} \, \prod_{j=1}^{N}\hat{f}(\xi_j)\,d\xi_1\dots d\xi_N\]

satisfies the estimate
\[
    \norm{T_a(f_1,\dots, f_N)}_{L^r}\leq C_{a,n} \norm{f_1}_{L^{q_1}}\dots\norm{f_N}_{L^{q_N}},
\]
provided that
\[
    m_j<\textarc{m}(1,\frac{p(\sum_{k=1}^{N}m_k)}{m_j},q_j),\quad {\rm for}\, j=1, \dots , N.
\]
In particular, if $a(x,\xi_1, \dots, \xi_N)$ verifies the estimate
\begin{equation}\label{L infty symbol estim}
\norm{\partial_{\xi_1}^{\alpha_1}\dots\partial_{\xi_N}^{\alpha_N} a(\cdot,\xi_1,\dots,\xi_N)}_{L^{\infty}}\leq C_{\alpha_{1}\dots \alpha_{N}} (1+|\xi_1|+\dots+|\xi_N|)^{m-\sum_{j=1}^{N}\abs{\alpha_j}},
\end{equation}
then $T_a$ is bounded from $L^{q_1} \times\dots\times L^{q_N} \to L^{r}$ provided that $\frac{1}{r}=\sum_{j=1}^{N}\frac{1}{q_j}$ and
\[
    m<-(n-1)\sum_{j=1}^{N}\abs{\frac{1}{q_j}-\frac{1}{2}}.
\]
Furthermore, $T_{a}$ with $a$ as in \eqref{L infty symbol estim} is bounded from $L^2 \times \dots\times L^2 \to L^{\frac{2}{N}}$ provided that $m<0$ and the phases $\varphi_{j}\in C^{\infty}(\R^n \times \R^n)$ are strongly non-degenerate and verify the condition $ |\partial_{x}^{\alpha} \partial^{\beta}_{\xi} \varphi_{j} (x, \xi)|\leq C_{j,\alpha,\beta}$ for $j=1,\dots, N$ and all multi-indices $\alpha $ and $\beta$ with $2\leq |\alpha|+|\beta|.$ Note in this case, we do not require any homogeneity from the phase functions.

\end{thm}
\begin{proof}
We will only give the proof of the theorem in the case of bilinear operators, since using the well-known inequality
\[\sum_{j\geq 0}\prod_{1\leq k\leq N} |a_{j,k}| \leq \prod_{1\leq k\leq N} \brkt{\sum_{j\geq 0} |a_{j,k}|^2}^{\frac{1}{2}}\]
and the H\"older inequality in \eqref{eq:bilinear} below yield the result in the multilinear case.

Let $\{\Psi_j\}_{j\geq 0}$ a Littlewood-Paley partition of unity in $\R^{2n}$ as in \eqref{eq:LittlewoodPaley}. Let for $j\geq 0$, $a_j(x,\xi, \eta)=a(x,\xi, \eta)\Psi_j(\xi, \eta)$.
Then we have
\[
    T_a(f,g)(x)=\sum_{j\geq 0} 2^{2jn} \iint a_{j}(x,2^{j}\xi, 2^{j} \eta) \widehat{f}(2^{j}\xi)\widehat{g}(2^{j}\eta)e^{i\varphi_{1}(x,2^{j}\xi)+i\varphi_{2}(x,2^{j}\eta)}\, d\xi \, d\eta.
\]
 Now since for any $j\geq 0$, $\Psi_j(2^j\xi, 2^{j} \eta)$ is supported in $B(0,2)\subset \T^{2n}$, following the argument in Theorem \ref{Linearlow} and expanding the amplitudes in Fourier series, we obtain
\[
    a_{j}(x,2^{j}\xi, 2^{j} \eta)=\sum_{(k,l)\in \Z^{2n}} a_{k,l}^j(x) e^{i\p{k,\xi}+i\p{l,\eta}}.
\]
Moreover, for all natural numbers $s\geq 1$, we have
\[
    \abs{a_{k,l}^j(x)}\lesssim \frac{1}{1+\abs{(k,l)}^s} \sum_{\alpha_1 + \alpha_{2}=\alpha;\,\abs{\,\alpha}\leq s}\int_{\T^{2n}} \abs{\partial^{\alpha_1}_\xi \partial^{\alpha_2}_\eta \brkt{a_{j}(x,2^j\xi, 2^{j} \eta)}}\, d\xi\, d\eta,
\]
and
\begin{equation}
\begin{split}
    \Vert a_{k,l}^j\Vert_{L^{p}}&\leq  \frac{1}{1+\abs{(k,l)}^s}  \sum_{\alpha_1 + \alpha_{2}=\alpha;\,\abs{\,\alpha}\leq s}\int_{\T^{2n}} \Vert \partial^{\alpha_1}_\xi \partial^{\alpha_2}_\eta a_{j}(x,2^j\xi, 2^{j} \eta)\Vert_{L^{p}_{x}}\, d\xi\, d\eta\\&\lesssim \frac{2^{j(m_1+m_2)}}{1+\abs{(k,l)}^s}.
\end{split}
\end{equation}
We now take a $\zeta\in C_{0}^{\infty}(\R^{n})$ equal to one in the cube $[-3,3]^n$ and such that $\supp \zeta\subset \T^{n}.$ This yields
\[
    \begin{split}
    T_a(f,g)(x)&=\sum_{(k,l)\in \Z^{2n}} \sum_{j\geq 0} a_{k,l}^j(x)\iint\zeta(2^{-j}\xi)\zeta(2^{-j}\eta) e^{i\p{2^{-j}k,\xi}}e^{i\p{2^{-j}l,\eta}}\widehat{f}(\xi)\widehat{g}(\eta)e^{i\varphi_{1}(x,\xi)+i\varphi_{2}(x,\eta)}\, d\xi\, d\eta\\
    &=\sum_{(k,l)\in \Z^{2n}} \sum_{j\geq 0} {\rm sgn\,}\brkt{a_{k,l}^{j}(x)}T_{\theta^{j,1}_{k,l},\varphi_1}(f)(x)T_{\theta^{j,2}_{k,l},\varphi_2}(g)(x),
    \end{split}
\]
where $\theta^{j,1}_{k,l}(x,\xi)=\abs{a_{k,l}^j(x)}^{\frac{m_1}{m_1+m_2}}\zeta(2^{-j}\xi)e^{i\p{2^{-j}k,\xi}},$  $\theta^{j,2}_{k,l}(x,\eta)=\abs{a_{k,l}^j(x)}^{\frac{m_2}{m_1+m_2}}\zeta(2^{-j}\eta)e^{i\p{2^{-j}l,\eta}}$ and ${\rm sgn\,}z=\frac{z}{\abs{z}}$ if $z\neq 0$ and zero elsewhere. If we let $R=\min (1,r),$ then the Cauchy-Schwarz and the H\"older inequalities yield
\begin{equation}\label{eq:bilinear}
    \norm{T_a(f,g)}_{L^r}^{R}\leq \sum_{(k,l)\in \Z^{2n}}
    \norm{\brkt{\sum_{j\geq 0} \abs{T_{\theta^{j,1}_{k,l},\varphi_1}(f)(x)}^2}^{\frac{1}{2}} }_{L^{r_1}}^{R}
    \norm{\brkt{\sum_{j\geq 0} \abs{T_{\theta^{j,2}_{k,l},\varphi_2}(g)(x)}^2}^{\frac{1}{2}}}_{L^{r_2}}^{R},
\end{equation}
where $\frac{1}{r}=\frac{1}{r_1}+\frac{1}{r_2}$, and
\[
    \frac{1}{r_1}=\frac{1}{q_1}+\frac{m_1}{p(m_1+m_2)},\qquad  \frac{1}{r_2}=\frac{1}{q_2}+\frac{m_2}{p(m_1+m_2)}.
\]
At this point we use Khinchin's inequality which yields that
\[
     \norm{\brkt{\sum_{j\geq 0} \abs{T_{\theta^{j,1}_{k,l},\varphi_1}(f)(x)}^2}^{\frac{1}{2}} }_{L^{r_1}(\R^n)}\lesssim
     \norm{\sum_{j\geq 0 } \varepsilon_j (t) T_{\theta^{j,1}_{k,l},\varphi_1}(f)(x)}_{L_{x,t}^{r_1}(\R^n\times [0,1])},
\]
where $\{\varepsilon_j (t)\}_j$ are the Rademacher functions.
Observe that the inner term is a linear FIO with the phase function $\varphi_1$ and the amplitude
\[
    \sigma^1_{k,l}(t,x,\xi)=\sum_{j\geq 0 } \varepsilon_{j}(t)\theta^{j,1}_{k,l}(x,\xi),\quad t\in [0,1],\, x,\xi\in \R^n.
\]
Picking $s_1,s_2$ such that $m_i<s_i<\textarc{m}(1,\frac{p(m_1+m_2)}{m_i},q_i)$, for $i=1,2,$ then since $\supp \zeta\subset B(0,2)$, one can see that for any multi-index $\alpha$
\[
    \abs{\d^\alpha_{\xi}\brkt{\varepsilon_j (t)\zeta(2^{-j}\xi)e^{i\p{2^{-j}k,\xi}}}}\lesssim \p{\xi}^{s_1-\abs{\alpha}}\brkt{1+\abs{k}^{\abs{\alpha}}}2^{-js_1},
\]
with a constant which is uniform in $j$ and $t$. In particular, $ \sigma^1_{k,l}\in L^{\frac{p(m_1+m_2)}{m_1}} S^{s_1}_1$ and
\[
    \norm{\d^\alpha_{\xi} \sigma^1_{k,l}(t,x,\xi)}_{L^{\frac{p(m_1+m_2)}{m_1}}}\lesssim \frac{\p{\xi}^{s_1-\abs{\alpha}}\brkt{1+\abs{k}^{\abs{\alpha}}}}{\brkt{1+\abs{k}^s}^{\frac{m_1}{m_1+m_2}}}.
\]
By the hypothesis on $m_1,m_2,$ we have that $\frac{p(m_1+m_2)}{m_1}\geq 2$, and therefore we can apply Theorem \ref{thm:p2} which yields
\[
    \norm{\sum_{j\geq 0 } \varepsilon_j (t) T_{\theta^{j,1}_{k,l},\varphi_1}(f)(x)}_{L^{r_1}(\R^n\times [0,1])}\lesssim \frac{1+\abs{k}^{\mu_1}}{\brkt{1+\abs{k}^s}^{\frac{m_1}{m_1+m_2}}} \norm{f}_{L^{q_1}},
\]
for a certain natural number $\mu_1.$
Arguing in the same way with the second term of \eqref{eq:bilinear} we have 
\[
    \norm{T_a(f,g)}_{L^r}^R\leq \sum_{(k,l)\in \Z^{2n}} \brkt{\frac{\brkt{1+\abs{k}^{\mu_1}}\brkt{1+\abs{l}^{\mu_2}}}{{1+\abs{(k,l)}^s}}}^R \norm{f}_{L^{q_1}}^R\norm{g}_{L^{q_2}}^R.
\]
Therefore by choosing $s$ large enough, we obtain the desired boundedness result.

The second part of the theorem concerning amplitudes satisfying the estimate \eqref{L infty symbol estim}, follows from our first result. Indeed if $p=\infty$ and $\varrho=1$ then $\textarc{m}(1,\infty ,q_j)= -(n-1)|\frac{1}{q_{j}}-\frac{1}{2}|$ and since according to Example \ref{interesting example}

\[ L^{\infty} S^{m}_1(n,N)\subset \bigcap_{m_1+\dots+m_N=m} L^{\infty}_\Pi S^{(m_1,\dots,m_N)}_{(1,\dots,1)}(n,N),\] for $m_j <0$, one can see using our previous claim concerning product type amplitudes, that the result follows provided $m<-(n-1)\sum_{j=1}^{N} |\frac{1}{q_{j}}-\frac{1}{2}|.$

The last assertion is a direct consequence of the method of proof of the first claim, and the $L^2$ boundedness of oscillatory integral operators with amplitudes in $S^{0}_{0,0}$ and strongly non-degenerate inhomogeneous phase functions satisfying the hypothesis of our theorem, which is due to K. Asada and D. Fujiwara \cite{AF}. The proof of the theorem is therefore concluded.
\end{proof}
\begin{rem}
The phase functions of the form $\langle x, \xi_{j} \rangle + \psi_{j}(\xi_j)$, $j=1,\dots, N$ lie in the realm of the above theorem, where different boundedness results apply. Cases of particular interest for the applications in nonlinear PDE's are:
\begin{enumerate}
\item $\psi_{j}(\xi_j)= |\xi_j|$ with $\xi_{j} \in \R^n$ (wave equation),
\item $\psi_{j}(\xi_j)= |\xi_j|^2$ with $\xi_{j} \in \R^n$ (Schr\"odinger equation),
\item $\psi_{j}(\xi_j)= \xi_j^3$ with $\xi_j \in \R$ (Korteweg-de Vries equation),
\item $\psi_{j}(\xi_j)= \langle \xi_j \rangle$ with $\xi_{j} \in \R^n$ (Klein-Gordon equation).
\end{enumerate}
\end{rem}

\end{document}